\documentclass[onefignum,onetabnum]{siamart171218}

\usepackage{graphicx}
\usepackage{booktabs}
\usepackage{amsmath}
\usepackage[english]{babel}
\usepackage[utf8]{inputenc}
\usepackage{lipsum}
\usepackage{amssymb}
\usepackage{multicol}
\usepackage[utf8]{inputenc}
\usepackage[english]{babel}
\usepackage{hyperref}
\usepackage{multirow}
\usepackage{cancel} 
\usepackage{bbm}
\usepackage{algorithm}
\usepackage{algorithmic}
\usepackage{enumitem}

\makeatletter
\AtBeginDocument{
  \def\refstepcounter@optarg[#1]#2{%
    \cref@old@refstepcounter{#2}%
    \cref@constructprefix{#2}{\cref@result}%
    \@ifundefined{cref@#1@alias}%
      {\def\@tempa{#1}}%
      {\def\@tempa{\csname cref@#1@alias\endcsname}}%
    \protected@edef\cref@currentlabel{%
      [\@tempa][\arabic{#2}][\cref@result]%
      \csname p@#2\endcsname\csname the#2\endcsname}}%
}
\makeatother

\newcommand{\bA}{{\bf A}}

\newcommand{\bD}{{\bf D}}

\newcommand{\bF}{{\bf F}}

\newcommand{\bH}{{\bf H}}
\newcommand{\bI}{{\bf I}}
\newcommand{\bJ}{{\bf J}}
\newcommand{\bK}{{\bf K}}
\newcommand{\bL}{{\bf L}}

\newcommand{\bP}{{\bf P}}

\newcommand{\bU}{{\bf U}}

\newcommand{\bb}{{\bf b}}

\newcommand{\bd}{{\bf d}}
\newcommand{\be}{{\bf e}}
\newcommand{\bff}{{\bf f}}
\newcommand{\bg}{{\bf g}}

\newcommand{\br}{{\bf r}}
\newcommand{\bs}{{\bf s}}
\newcommand{\bt}{{\bf t}}

\newcommand{\bx}{{\bf x}}
\newcommand{\by}{{\bf y}}

\newcommand{\bSigma}{{\bf \Sigma}}

\newcommand{\calF}{\mathcal{F}}

\newcommand{\calN}{\mathcal{N}}

\newcommand{\calP}{\mathcal{P}}

\newcommand{\calR}{\mathcal{R}}

\newcommand{\bbR}{\mathbb{R}}

\newsiamremark{remark}{Remark}

\newcommand{\TheTitle}{Variable Projection Methods for Solving Regularized Separable Inverse Problems with Applications to Semi-Blind Image Deblurring} 
\newcommand{\ShortTitle}{VarPro for separable nonlinear inverse problems} 
\newcommand{\TheAuthors}{D. B. Comerso Salzer, M. I. Espa\~ nol, and G. Jeronimo}

\headers{\ShortTitle}{\TheAuthors}
\title{{\TheTitle}}

\author{
Delfina B. Comerso Salzer\thanks{Departamento de Matem\'atica, Facultad de Ciencias Exactas y Naturales, Universidad de Buenos Aires, Argentina.}
\and
Malena I. Espa\~nol\thanks{School of Mathematical and Statistical Sciences, Arizona State University, Tempe, AZ, United States (\email{malena.espanol@asu.edu}).}
\and
Gabriela Jeronimo\thanks{CONICET–Universidad de Buenos Aires, Instituto de Investigaciones Matem\'aticas ``Luis A. Santal\'o'' (IMAS), Argentina (\email{jeronimo@dm.uba.ar}).}
}

\begin{document}
\maketitle

\begin{abstract} Separable nonlinear least squares problems appear in many inverse problems, including semi-blind image deblurring. The variable projection (VarPro) method provides an efficient approach for solving such problems by eliminating linear variables and reducing the problem to a smaller, nonlinear one. In this work, we extend VarPro to solve minimization problems containing a differentiable regularization term on the nonlinear parameters, along with a general-form Tikhonov regularization term on the linear variables. Furthermore, we develop a quasi-Newton method for solving the resulting reduced problem, and provide a local convergence analysis under standard smoothness assumptions, establishing conditions for superlinear or quadratic convergence. For large-scale settings, we introduce an inexact LSQR-based variant and prove its local convergence despite inner-solve and Hessian approximations. Numerical experiments on semi-blind deblurring show that parameter regularization prevents degenerate no-blur solutions and that the proposed methods achieve accurate reconstructions, with the inexact variant offering a favorable accuracy-cost tradeoff consistent with the theory. 

\end{abstract}



\section{Introduction} We are interested in solving discrete ill-posed inverse problems of the form 
\begin{equation*}
\bA(\by)\bx  \approx \bb =  \bb_{\rm true} + \mathbf{\epsilon} \quad \mbox{ with } \bA( \by_{\rm true}) \bx_{\rm true} = \bb_{\rm true},
\end{equation*}
where the vector $\bb_{\rm true} \in \mathbb{R}^m$ denotes the unknown noise-free data vector while $\mathbf{\epsilon} \in \mathbb{R}^m$ represents additive measurement errors. The matrix $\bA(\by) \in \mathbb{R}^{m\times n}$ with $m\geq n$ models a forward operator and is typically severely ill-conditioned. Although the forward operator is not known explicitly, we assume that it can be represented as a differentiable function of a small number of parameters, collected in a vector $\by \in \mathbb{R}^r$, where $r \ll n$, and that the resulting mapping $\by \mapsto \bA(\by)$ is differentiable. Our objective is to recover accurate approximations of both the solution $\bx_{\rm true}$ and the parameter vector $\by_{\rm true}$ using the measured data $\bb$ and a matrix function that maps the unknown vector $\by$ to an $m \times n$ matrix $\bA$. To accomplish this task, we could solve
\begin{equation}\label{eq: nsl2}
\min_{\bx, \by} \frac{1}{2}\left\|\bA(\by)\bx - \bb \right\|^2 + \frac{\lambda^2}{2} \left\|\bL \bx\right\|^2,
\end{equation}
where $\|\cdot\|$ denotes the 2-norm (which will also be used for matrices throughout the paper, in which case it denotes the matrix norm induced by the vector 2-norm), $\lambda>0$ is a \emph{regularization parameter}, and $\bL \in \mathbb{R}^{q\times n}$ is a \emph{regularization operator}. We assume that $\bL$ satisfies $\mathcal{N}(\bA(\by)) \cap \mathcal{N} (\bL) = \{\boldsymbol{0}\}$ for all feasible values of $\by$, so that the minimization problem \eqref{eq: nsl2} has a unique solution for $\by$ fixed. We call problems of the form \eqref{eq: nsl2}  \emph{separable} nonlinear least squares (SNLS) since the observations depend nonlinearly on the vector of unknown parameters $\by$ and linearly on the solution vector~$\bx$. 

The Variable Projection (VarPro) method was originally developed in the 1970s by Golub and Pereyra \cite{golub1973differentiation} to solve \eqref{eq: nsl2} for unregularized separable nonlinear least squares problems (i.e., $\lambda=0$) and has been widely recognized for its efficiency in solving these problems. The main idea behind VarPro is to eliminate the linear variables $\mathbf{x}$ through projection and to reduce the original problem to a smaller nonlinear least squares problem in the parameters $\mathbf{y}$. This reduced nonlinear least squares problem can be solved using the Gauss-Newton Method.

In \cite{Espanol_2023}, Espa\~nol and Pasha extended VarPro to solve inverse problems with general-form Tikhonov regularization for general matrices $\mathbf{L}$. They named this method GenVarPro. For special cases where computing the generalized singular value decomposition (GSVD) of the pair $\{\mathbf{A}(\mathbf{y}),\mathbf{L}\}$ for a fixed value of $\mathbf{y}$ is feasible or a joint spectral decomposition exists, they provided efficient ways to compute the Jacobian matrix and the solution of the linear subproblems. For large-scale problems, where matrix decompositions are not an option, they proposed computing a reduced Jacobian and applying projection-based iterative methods and generalized Krylov subspace methods to solve the linear subproblems.
Following on this theme, Espa\~nol and Jeronimo introduced the Inexact-GenVarPro in \cite{espanol2025local}, which considers a new approximate Jacobian where iterative methods such as LSQR and LSMR are used to solve the linear subproblems. Furthermore, specific stopping criteria were proposed to ensure Inexact-GenVarPro's local convergence.

In this work, we will show how to extend GenVarPro and Inexact-GenVarPro to solve
\begin{equation}\label{eq: nlls_regy}
\min_{\mathbf{x}, \mathbf{y}} \frac{1}{2}\left\|\mathbf{A}(\mathbf{y})\mathbf{x} - \mathbf{b} \right\|^2 + \frac{\lambda^2}{2} \left\|\mathbf{L} \mathbf{x}\right\|^2 + \mathcal{R}(\mathbf{y}),
\end{equation}
where $\mathcal{R}(\mathbf{y})$ plays the role of regularization on the parameter vector~$\mathbf{y}$. We will refer to this problem as the \emph{regularized} SNLS problem (RSNLS). We will motivate the need to incorporate this regularization term in $\mathbf{y}$ in the context of a semi-blind image deblurring problem by showing some examples where, without it, the solution of the reduced problem is trivial (e.g., $\mathbf{y}=\boldsymbol{0}$, $\mathbf{A}(\mathbf{y})$ becomes the identity matrix, and $\mathbf{x}=\mathbf{b}$). 
This situation is similar to what the authors in~\cite{levin2009understanding} refer to as the \emph{no-blur solution} in Bayesian blind deconvolution, where the naive MAP estimate corresponds to a Dirac delta kernel and the observed image itself as the sharp estimate. A related issue was also observed in~\cite{cornelio2014constrained} for the semi-blind deblurring problem, where the authors proposed adding a nonnegativity constraint on the linear variables in problem~\eqref{eq: nsl2}, rather than introducing a regularization term on the nonlinear variables. 
In our setting, regularizing $\mathbf{y}$ prevents such trivial solutions by penalizing degenerate blurring parameters.

Similar variational formulations have been introduced in recent works across a range of applications. These include the training of neural networks~\cite{newman2021train} and the learning of optimal linear mappings~\cite{newman2025optimal} (where the corresponding optimization problems were solved using VarPro) as well as computed tomographic reconstruction~\cite{uribe2022hybrid} (where an alternating minimization scheme was used). More general extensions that go beyond classical Tikhonov regularization and are solved using VarPro-based approaches have been developed in~\cite{aravkin2017efficient, van2021variable}, addressing PDE-constrained optimization problems (such as boundary control and optimal transport) and robust dynamic inference. In the recent work~\cite{marmin2025varprox}, the authors propose a primal-dual variable projection algorithm for solving~\eqref{eq: nlls_regy} with non-smooth regularization term $\mathcal{R}(\mathbf{y})$ and box constraints, motivated by an image texture analysis problem in which $\by$ is high-dimensional.

In this paper, we use semi-blind image deblurring to motivate our extension of the VarPro method and to demonstrate its practical effectiveness. In the image deblurring problem~\cite{austin2022image, espanol2009multilevel, deblurring}, the goal is to recover the unknown sharp image $\bx_{\rm true}$ from the observed blurred and noisy image $\bb$. The blurring operator $\bA$ may be fully known (non-blind deblurring), completely unknown (blind deblurring), or known only up to a low-dimensional parametrization (semi-blind deblurring). In blind deblurring, many methods incorporate explicit regularization on the blur operator or kernel to avoid degenerate solutions~\cite{chan1998total, dykes2021lanczos, levin2009understanding}. In contrast, for semi-blind deconvolution problems, regularization on the blur parameters $\by$ is typically not included~\cite{buccini2018semiblind, chung2010efficient, cornelio2014constrained, espanol2025local, Espanol_2023, gazzola2021regularization}. As we show in this work, the absence of such parameter regularization can lead to trivial or unstable solutions, and incorporating an appropriate regularizer on $\by$ is essential for achieving meaningful reconstructions.

The paper is organized as follows. In Section \ref{sec:varpro}, we review the VarPro method, give a motivating example of why we need to incorporate a regularization term for the model parameters, and present an extension of VarPro for RSNLS problems. In Section \ref{section:quasi-Newton}, we present a quasi-Newton method to solve the reduced problem, and introduce a new method, which we call \texttt{RGenVarPro}, and prove its local convergence. In Section \ref{section:iRGenVarPro}, we introduce the inexact version of \texttt{RGenVarPro} named \texttt{iRGenVarPro}, for solving large-scale problems, where the solution of a linear subproblem is found using LSQR. A local convergence analysis is also presented for this method. In Section \ref{section:special_cases}, we introduce the specific regularization terms used in the numerical experiments in Section \ref{section:NumericalExamples}. These experiments verify our convergence results and the effectiveness of the proposed methods. The conclusions follow in Section \ref{section: conclusions}.

\section{Variable Projection Method}\label{sec:varpro}
The basic idea of VarPro is to solve a minimization problem  
$$\min_{\bx,\,\by}\,\calF(\bx,\by)$$
by observing that, for a fixed \(\by\), we can define
$$\bx(\by) = \arg\min_{\bx}\,\calF(\bx, \by)$$
and substituting $\bx(\by)$ into $\calF$ yields the reduced problem  
\begin{equation}\label{eq: minvarphi}
\min_{\by}\,\varphi(\by)
\end{equation}
 with $$\varphi(\by) = \calF(\bx(\by), \by).$$
 
This approach effectively eliminates the variables $\bx$ from the formulation, resulting in a reduced functional $\varphi(\by)$ that depends solely on $\by$.

\subsection{Motivating Example}\label{sec: motivating example}
We now present an example to motivate the need for the regularization term $\mathcal{R}(\by)$ in \eqref{eq: nlls_regy}, even if $\by$ is of small size. We will show that, without the regularization term, the solution of the minimization problem is trivial (i.e., $\mathbf{y}=\boldsymbol{0}$, and therefore $\mathbf{A}(\mathbf{y})$ becomes the identity matrix, and the solution $\bx=\bb$). 

Our example is a one-dimensional semi-blind deconvolution problem, similar to that considered in the numerical experiments of~\cite{Espanol_2023}, which consists of solving the minimization problem
\begin{equation}\label{eq:snls}
\min_{\bx, \sigma}  \calF(\bx,\sigma) 
\end{equation}
with $$\calF(\bx,\sigma) = \frac{1}{2}\| \bA(\sigma)\bx - \bb \|^2 + \frac{\lambda^2}{2} \|\bx\|^2.$$
In this example, $\by = \sigma$ is a scalar parameter, which allows for an easier analysis of the behavior of the reduced functional $\varphi\colon \mathbb{R} \to \mathbb{R}$. The matrix $\bA(\sigma)\in \mathbb{R}^{n\times n}$ is defined by the Gaussian kernel
$g_\sigma(s) = \mbox{exp}\left(-\frac{s^2}{2\sigma^2}\right)$, where $\sigma$ is the parameter defining the ``width'' of the Gaussian. Assuming zero boundary conditions on $\bx$, $\bA(\sigma)$ is a symmetric Toeplitz matrix with its first column defined by
\begin{align}
   [\bA(\sigma)]_{j+1,1}= \frac{a_j(\sigma)}{G_0(\sigma)}  , \qquad j=0, \dots, n-1,\nonumber
\end{align}
where $a_j(\sigma)=g_\sigma(j)$ and  $G_0(\sigma) = \sum_{j=0}^{n-1} a_j(\sigma)$. Note that $a_0(\sigma)=1$.

We now apply VarPro to solve \eqref{eq:snls}. For every $\sigma$, we define the vector $\bx(\sigma)$ as the solution of
\[ \min_{\bx} \frac{1}{2}\|\bA(\sigma) \bx - \bb \|^2 + \frac{\lambda^2}{2}\| \bx\|^2,\]
which has the closed form
\[\bx(\sigma) = (\bA(\sigma)^\top \bA(\sigma) + \lambda^2 \bI)^{-1} \bA(\sigma)^\top \bb.\]
Next, we  replace this expression  in \eqref{eq:snls}, and solve the reduced problem
\begin{equation*}
\min_\sigma  \varphi(\sigma) 
\end{equation*}
with 
\begin{equation}\label{eq: Tik-problem}
\varphi(\sigma) = \frac{1}{2}\|\bA(\sigma) \bx(\sigma) - \bb\|^2 + \frac{\lambda^2}{2} \| \bx(\sigma)\|^2.
\end{equation}

In the numerical examples in \cite{Espanol_2023}, the convergence curves for $\sigma$ show a semiconvergence behavior: VarPro approaches the true value of $\sigma$, but then appears to converge to a different value. Here, we revisit that example from a different perspective. For $n=128$, we consider two different instances of $\bx_{\rm{true}}$, one smooth and one with edges. Figure \ref{fig:varpro-fminsearch} illustrates the behavior of the reduced functional $\varphi(\sigma)$ for these two cases and for different values of $\lambda$. In both cases, the minimizer occurs at $\sigma=0$. The figure also shows iterates produced by a MATLAB optimization solver, which converge to this minimizer.
\begin{figure}[htbp]
    \centering
    \includegraphics[width=\textwidth]{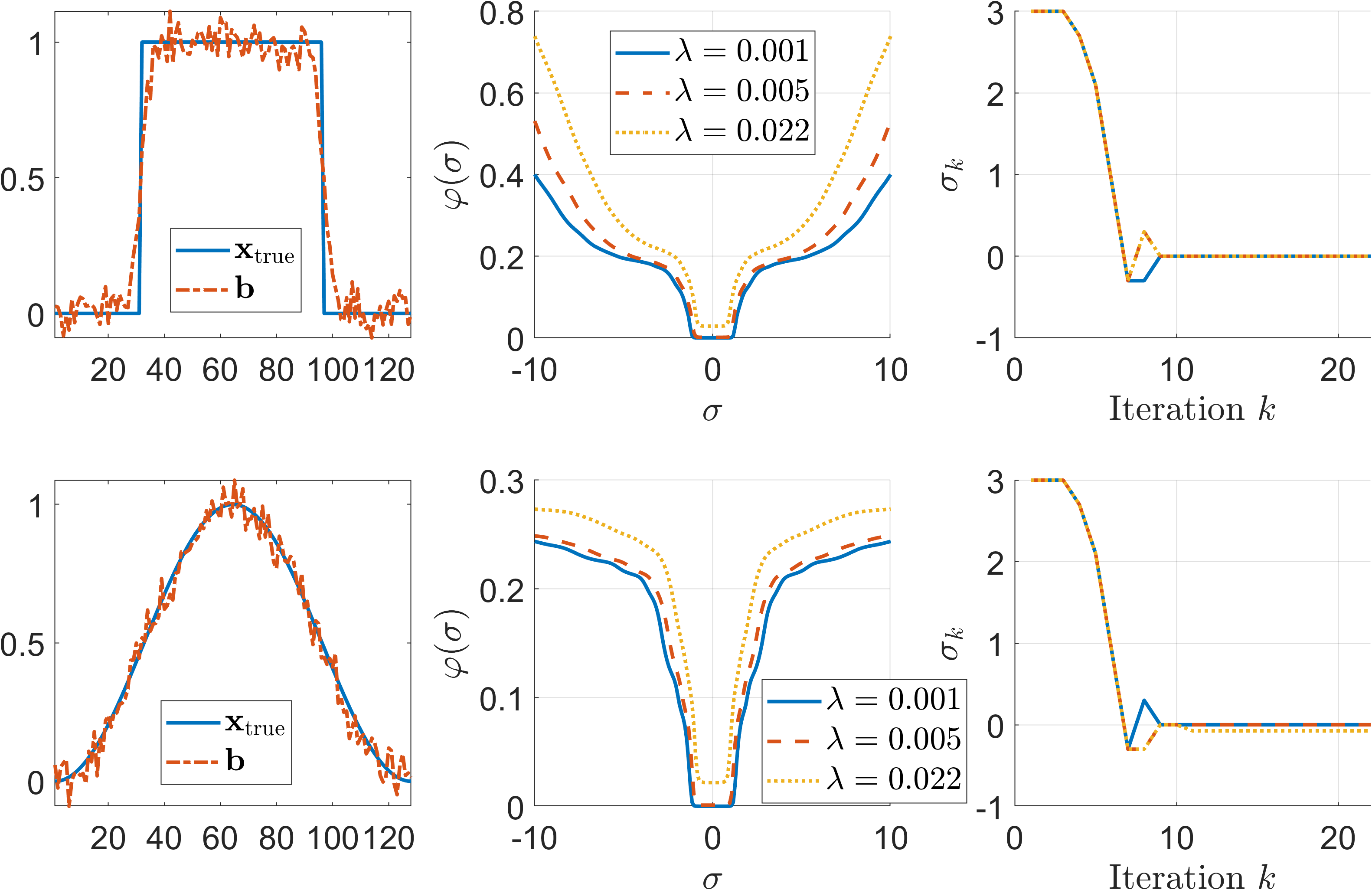}
    \caption{(Left) True signals: an edgy signal (top) and a smooth signal (bottom), and their corresponding blurred and noisy signals; (Middle) reduced functional $\varphi(\sigma)$ for selected values of $\lambda$; and (Right) convergence of $\sigma_k$ using the built-in MATLAB function \texttt{fminsearch}.}
    \label{fig:varpro-fminsearch}
\end{figure}

Next, we want to show analytically for $n=2$ that the minimizer of \eqref{eq:snls} is in fact $\sigma=0$. 
We start by rewriting the matrix $\bA(\sigma)$:
\begin{equation*}\bA(\sigma) = \frac{1}{a_0(\sigma)+a_1(\sigma)}\left[ \begin{array}{cc} a_0(\sigma) & a_1(\sigma) \\ a_1(\sigma) & a_0(\sigma)\end{array}\right] =  \frac{1}{1+a_1(\sigma)}\left[ \begin{array}{cc} 1 & a_1(\sigma) \\ a_1(\sigma) & 1\end{array}\right]. \end{equation*} The singular value decomposition (SVD) of the symmetric matrix $\bA(\sigma)$ is given by 
\[\bA(\sigma) = \bU \bSigma (\sigma) \bU^\top\] 
with 
\begin{equation*}\bU = \frac{1}{\sqrt{2}}\left[\begin{array}{cc} 1 & 1 \\ 1 & -1\end{array} \right] \mbox{ and } \bSigma (\sigma) = \left[\begin{array}{cc} 1 & 0 \\ 0 & \frac{1-a_1(\sigma)}{1+a_1(\sigma)}\end{array} \right] = \left[\begin{array}{cc} \mu_0(\sigma) & 0 \\ 0 & \mu_1(\sigma)\end{array} \right].\end{equation*}
To simplify notation, we write $\Sigma(\sigma)=\Sigma$.
Using the SVD of $\bA(\sigma)$ we can write
\begin{equation*}\bx(\sigma) = \bU (\bSigma^2 + \lambda^2 \bI)^{-1} \bSigma \bU^{\top} \bb.\end{equation*}
and use this expression of the solution $\bx(\sigma)$ in \eqref{eq: Tik-problem}. Then, defining $\widetilde{\bb} = \bU^\top \bb$, we have that
\begin{align*}
\| \bA(\sigma) \bx(\sigma) - \bb\|^2 = 
\| \bU (\bSigma (\bSigma^2 + \lambda^2 \bI)^{-1} \bSigma  - \bI) \bU^{\top} \bb\|^2= \| (\bSigma (\bSigma^2 + \lambda^2 \bI)^{-1} \bSigma  - \bI) \widetilde{\bb}\|^2, 
\end{align*}
where the last equality is a consequence of $\bU$ being an orthogonal matrix. Similarly,
\[ \| \bx(\sigma)\|^2 = \| (\bSigma^2 + \lambda^2 \bI)^{-1} \bSigma \widetilde{\bb} \|^2.\]
By explicit computations, 
\[ (\bSigma^2 + \lambda^2 \bI)^{-1} \bSigma = \left[ \begin{array}{cc} \frac{1}{1+\lambda^2} & 0 \\  0 & \frac{\mu_1(\sigma)}{\mu_1(\sigma)^2+\lambda^2} \end{array}\right]\]
and
\[ \bSigma (\bSigma^2 + \lambda^2 \bI)^{-1} \bSigma  - \bI =  \left[ \begin{array}{cc} \frac{-\lambda^2}{1+\lambda^2} & 0 \\  0 & \frac{-\lambda^2}{\mu_1(\sigma)^2+\lambda^2} \end{array}\right].\] 
Therefore, if $\widetilde{\bb} = \left[ \begin{array}{c} \beta_1 \\ \beta_2\end{array}\right]$, we obtain
the following expression for the function $\varphi(\sigma)$ we want to minimize:

\begin{align*} 
\varphi(\sigma) &= \frac{1}{2} \| \bA(\sigma) \bx(\sigma) - \bb\|^2 + \frac{\lambda^2}{2} \| \bx(\sigma)\|^2 \\
&= \frac{1}{2} \left(\frac{\lambda^4}{\left(1 +\lambda^2\right)^2} \beta_1^2 +  \frac{\lambda^4}{\left(\mu_1(\sigma)^2+\lambda^2\right)^2} \beta_2^2 \right) \\ & \quad + \frac{1}{2}\Bigg(\frac{\lambda^2}{\left(1 +\lambda^2\right)^2} \beta_1^2 +  \frac{\lambda^2\mu_1(\sigma)^2}{\left(\mu_1(\sigma)^2+\lambda^2\right)^2} \beta_2^2 \Bigg)\\
&= \frac{1}{2}\frac{\lambda^2}{1 +\lambda^2} \beta_1^2 +  \frac{1}{2}\frac{\lambda^2}{\mu_1(\sigma)^2 +\lambda^2} \beta_2^2.   
\end{align*}

Note that by defining $a_1(0) = 0$, we have that $a_1(\sigma)$ is differentiable and 
\[a_1'(\sigma) = \frac{1}{\sigma^3}\mbox{exp}\left(-\frac{1}{2 \sigma^2}\right) \hbox{ if }  \sigma\ne 0  \ \hbox{ and } a_1'(0) = 0.\]
Then, we can differentiate $\mu_1'(\sigma)$ and its derivative is given by 
\[
\mu_1'(\sigma)= \frac{-2a_1'(\sigma)}{(1+a_1(\sigma))^2}.
\]
Now, we can differentiate $\varphi(\sigma)$ and obtain
\[
\varphi'(\sigma)
=
-\frac{\lambda^2\,\mu_1(\sigma)\mu_1'(\sigma)}
{\left(\mu_1(\sigma)^2+\lambda^2\right)^2}\,\beta_2^2.
\]
Assume $\beta_2 \ne 0$ (otherwise, $\varphi(\sigma)$ is constant). Then the sign of
$\varphi'(\sigma)$ is determined by the sign of $\mu_1(\sigma)\mu_1'(\sigma)$.
Using the expression above,
\[
\mu_1(\sigma)\mu_1'(\sigma)
=
\frac{-2\left(1-a_1(\sigma)\right)a_1'(\sigma)}{(1+a_1(\sigma))^3}.
\]
Since $0\le a_1(\sigma)<1$ for every $\sigma$, we conclude that the zeros of $\varphi'(\sigma)$ are those of $a_1'(\sigma)$; then, $\varphi$ has a unique critical point at $\sigma=0$. It is clear from the expression above that, for every $\sigma$, the sign of $\varphi'(\sigma)$ coincides with the sign of $a_1'(\sigma)$. Noticing that $a_1'(\sigma) <0$ if $\sigma<0$ and $a_1'(\sigma)>0$ if $\sigma>0$, we conclude that $\varphi$ attains its minimum at $\sigma=0$.

This explicit $2\times2$ calculation illustrates the phenomenon observed in our
numerical experiments: without regularization of the blur parameter, the reduced
problem can (and, for typical data, does) favor the degenerate width $\sigma=0$.
Importantly, the $2\times2$ matrix $\bA$ in this example is not only symmetric
Toeplitz, but also circulant, and its diagonalizing matrix $\bU$ is precisely the
$2\times2$ discrete Fourier transform. This structural observation motivates an
extension of the analysis to problems with periodic boundary conditions in one and
two dimensions, where the system matrices are likewise diagonalizable by the
discrete Fourier transform \cite{deblurring}. We develop this general framework in
Appendix~\ref{appendix}. 

The numerical experiments reported in Section~\ref{section:NumericalExamples} make direct use of the
two-dimensional version of this analysis. In that setting, periodic boundary
conditions lead to a block-circulant with circulant blocks (BCCB) matrix
$\bA(\sigma)$ \cite{deblurring}, which is diagonalizable by the two-dimensional
discrete Fourier transform. The theoretical analysis in Appendix~\ref{appendix}
is carried out for the case $\bL=\bI$ in order to isolate and highlight the
behavior of the reduced functional with respect to the blur parameter $\sigma$.
In the numerical experiments, however, we employ a Laplacian-based regularization
operator $\bL$ to promote smoothness in the reconstructed solution. While this
choice changes the structure of the regularization term, the qualitative
behavior predicted by the analysis is still observed.

\subsection{VarPro for Regularized Separable Nonlinear Least Squares Problems}
For our RSNLS problem \eqref{eq: nlls_regy}, the function that defines the original minimization problem is
\begin{equation}\label{eq: def_calF}
\calF(\bx, \by)= \frac{1}{2}\| \bK(\by) \bx - \bd \|^2 + \mathcal{R}(\by),
\end{equation}
where
\begin{equation}\label{eq: defKd}
\bK(\by)=\left[\begin{array}{c} \bA(\by) \\ \lambda \bL\end{array}\right], \quad \bd=\left[\begin{array}{c} \bb \\ \boldsymbol{0} \end{array}\right], \end{equation}
and $\mathcal{R}(\by)$ is a continuous differentiable function to be chosen.
Under the assumption $\calN(\bA(\by))\cap \calN(\bL) = \{\boldsymbol{0}\}$, it follows that $\bK(\by)$ has maximal rank, and so, for a fixed $\by$,
\[\bx(\by)=\arg\min_{\bx}  \frac{1}{2}\| \bK(\by) \bx - \bd \|^2 + \mathcal{R}(\by) = \arg\min_{\bx}  \frac{1}{2}\| \bK(\by) \bx - \bd \|^2,\] 
and we have an explicit closed formula:
\[ \bx(\by) = \bK(\by)^{\dagger} \bd,\]
where $\bK(\by)^{\dagger}= (\bK(\by)^\top\bK(\by))^{-1}\bK^\top(\by)$ is the Moore-Penrose pseudoinverse of $\bK(\by)$.
We then obtain that
\[ \varphi(\by) =  \calF(\bK(\by)^{\dagger} \bd, \by) = \frac{1}{2}\| \bK(\by) \bK(\by)^{\dagger} \bd - \bd \|^2  + \mathcal{R}(\by) =  \frac{1}{2}\|  -\calP_{\bK(\by)}^\perp \bd \|^2 + \mathcal{R}(\by),\]
where $\mathcal{P}_{\bK(\by)}^\perp = \mathbf{I} - \bK(\by)\bK(\by)^\dagger$ is the orthogonal projector onto the null space of $\bK(\by)^\top$. If we define $$\bff(\by)=  -\calP_{\bK(\by)}^\perp \bd,$$ then we have
\begin{equation}\label{eq: def_varphi}
 \varphi(\by) = \dfrac{1}{2}\| \bff(\by) \|^2 + \mathcal{R}(\by). 
\end{equation}

The following theorem shows that the variable projection approach allows us to solve the original minimization problem. Its proof is based on the original proof from~\cite{golub1973differentiation}, which we adapt to our setting.

\begin{theorem} Assume that $\calN(\bA(\by))\cap \calN(\bL) = \{\boldsymbol{0}\}$ for every $\by$ in a domain $\Omega\subset \mathbb{R}^r$ and that the map $\by \mapsto \bA(\by)$ is differentiable in $\Omega$. 
    Let $\calF(\bx,\by)$ and $\varphi(\by)$ be the functions introduced in \eqref{eq: def_calF} and \eqref{eq: def_varphi}, respectively. 
    \begin{enumerate}
        \item[(1)] If $(\bx^*, \by^*)$ is a global minimizer of $\calF(\bx, \by)$ in $\bbR^n \times \Omega$, then $\bx^* = \bK(\by^*)^{\dagger}\bd$ and $\by^*$ is a global minimizer of $\varphi(\by)$ in $\Omega$.
        \item[(2)] If $\by^*$ is a critical point (resp.~a global minimizer) of $\varphi(\by)$ in $\Omega$ and $\bx^*= \bK(\by^*)^{\dagger}\bd$, then $(\bx^*, \by^*)$ is a critical point (resp.~a global minimizer) of $\calF(\bx, \by)$ in $\bbR^n \times \Omega$.
    \end{enumerate}
\end{theorem}

\begin{proof}
For every $\by \in \Omega$, because of the assumption $\calN(\bA(\by))\cap \calN(\bL) = \{\boldsymbol{0}\}$, we have  
\[\bK(\by)^{\dagger} \bd = \arg\min_\bx \frac{1}{2}\| \bK(\by) \bx - \bd\|^2.\]
This implies that, for $\by\in \Omega$, the following inequality holds for every $\bx\in \bbR^n$: 
\begin{equation}\label{eq: varphi-calF}
\varphi(\by) = \dfrac{1}{2} \| \bK(\by) \bK(\by)^{\dagger} \bd - \bd\|^2 +\calR(\by)\le \dfrac{1}{2} \| \bK(\by) \bx- \bd\|^2 +\calR(\by) = \calF(\bx, \by),
\end{equation}
and the inequality is strict for $\bx \ne \bK(\by)^{\dagger} \bd$.

\medskip
\noindent (1) Assume $(\bx^*, \by^*)$ is a global minimizer of $\calF(\bx, \by)$ in $\bbR^n \times \Omega$. By the inequality above, $\varphi(\by^*)\le \calF(\bx^*, \by^*)$. Now, if $\bx^{\diamond}= \bK(\by^*)^{\dagger}\bd$, from the definition of $\varphi$ and the fact that $(\bx^*, \by^*)$ is a global minimizer of $\calF(\bx, \by)$, we have
$\varphi(\by^*) = \calF(\bx^{\diamond}, \by^*) \ge \calF(\bx^*, \by^*)$. Then, the equality $\varphi(\by^*) = \calF(\bx^*, \by^*)$ holds and, consequently, $\bx^* = \bK(\by^*)^{\dagger}\bd$.

Therefore, for every $\by_0\in \Omega$, if $\bx_0= \bK(\by_0)^\dagger \bd$, we have $\varphi(\by_0) = \calF(\bx_0, \by_0) \ge \calF(\bx^*, \by^*) = \varphi(\by^*)$, which implies that $\by^*$ is a global minimizer of $\varphi(\by)$ in $\Omega$.

\medskip
\noindent (2) Let $\by^*\in \Omega$ be a critical point of $\varphi(\by)$ and $\bx^* = \bK(\by^*)^{\dagger} \bd$. In order to show that $(\bx^*, \by^*)$ is a critical point of $\calF$, we compute the gradient $\nabla \calF(\bx^*, \by^*)$. 
 First, notice that differentiating with respect to the variables $\bx$ we obtain:
\[\nabla_\bx \calF(\bx^*, \by^*) = \bK(\by^*)^\top (\bK(\by^*) \bx^* - \bd) = - \bK(\by^*)^\top \calP_{\bK(\by^*)}^\perp \bd = \boldsymbol{0}.\]
Now, for $j=1,\dots, r$,
\[\dfrac{\partial \calF}{\partial y_j} (\bx, \by)  = \left(\dfrac{\partial \bK}{\partial y_j}(\by)\bx\right)^\top (\bK(\by)\bx- \bd)+ \dfrac{\partial \calR}{\partial y_j}(\by).\]
Taking into account that $\bx^* = \bK(\by^*)^{\dagger} \bd$, we obtain
\begin{equation}\label{eq: diffcalF}
\dfrac{\partial \calF}{\partial y_j} (\bx^*, \by^*) = -(\bK(\by^*)^{\dagger} \bd )^\top \left(\dfrac{\partial \bK}{\partial y_j}(\by^*)\right)^\top \calP_{\bK(\by*)}^\perp \bd + \dfrac{\partial \calR}{\partial y_j}(\by^*). 
\end{equation}

On the other hand,
\[\dfrac{\partial \varphi (\by) }{\partial y_j} = \left( \dfrac{\partial \bff(\by)}{\partial y_j}  \right)^\top \bff(\by) + \dfrac{\partial \calR (\by)}{\partial y_j}, \]
where $\bff(\by) = -\calP_{\bK(\by)}^\perp \bd$.
According to \cite{Espanol_2023}, 
\begin{equation}\label{eq: Jac_columns}
    \dfrac{\partial \bff(\by)}{\partial y_j} =  \left(\calP_{\bK(\by)}^\perp \dfrac{\partial \bK(\by)}{\partial y_j} \bK(\by)^\dagger + (\calP_{\bK(\by)}^\perp \dfrac{\partial \bK(\by)}{\partial y_j} \bK(\by)^\dagger)^\top \right) \bd.\end{equation}
Therefore,
\begin{align*}
\dfrac{\partial \varphi (\by) }{\partial y_j} & = - \bd^\top \left((\calP_{\bK(\by)}^\perp \dfrac{\partial \bK(\by)}{\partial y_j} \bK(\by)^\dagger)^\top+\calP_{\bK(\by)}^\perp \dfrac{\partial \bK(\by)}{\partial y_j} \bK(\by)^\dagger  \right) \calP_{\bK(\by)}^\perp \bd +\dfrac{\partial \calR (\by)}{\partial y_j} \\
&= - \bd^\top (\calP_{\bK(\by)}^\perp \dfrac{\partial \bK(\by)}{\partial y_j} \bK(\by)^\dagger)^\top \calP_{\bK(\by)}^\perp \bd +\dfrac{\partial \calR (\by)}{\partial y_j} 
\end{align*}
since $\bK(\by)^\dagger  \calP_{\bK(\by)}^\perp =0$. By computing the transpose in the first term, and using that $(\calP_{\bK(\by)}^\perp)^\top = \calP_{\bK(\by)}^\perp$ and $(\calP_{\bK(\by)}^\perp)^2 =\calP_{\bK(\by)}^\perp$, 
\begin{align}\label{eq: diffvarphi}
\dfrac{\partial \varphi (\by) }{\partial y_j} &= - \bd^\top (\bK(\by)^\dagger)^\top  \left( \dfrac{\partial \bK(\by)}{\partial y_j} \right)^\top(\calP_{\bK(\by)}^\perp)^\top \calP_{\bK(\by)}^\perp \bd +\dfrac{\partial \calR (\by)}{\partial y_j}   \nonumber \\ 
&=  - (\bK(\by)^\dagger \bd)^\top \left( \dfrac{\partial \bK(\by)}{\partial y_j} \right)^\top \calP_{\bK(\by)}^\perp \bd +\dfrac{\partial \calR (\by)}{\partial y_j}  .
\end{align}
From \eqref{eq: diffcalF} and \eqref{eq: diffvarphi}, since $\by^*$ is a critical point of $\varphi(\by)$, it follows that
\[\dfrac{\partial \calF}{\partial y_j} (\bx^*, \by^*)=\dfrac{\partial \varphi} {\partial y_j}  (\by^*) =0. \]
We conclude that $\nabla \calF(\bx^*, \by^*)=0$, that is, $(\bx^*, \by^*)$ is a critical point of $\calF$.

If $\by^*$ is a global minimizer of $\varphi(\by)$ in $\Omega$, for $(\bx_0, \by_0)\in \bbR^n\times \Omega$, by using inequality \eqref{eq: varphi-calF} we deduce that $\calF(\bx_0, \by_0)\ge \varphi(\by_0) \ge\varphi(\by^*) = \calF(\bx^*, \by^*)$. Therefore, $(\bx^*, \by^*)$ is a global minimizer of $\calF(\bx, \by)$ in $\bbR^n\times \Omega$.
\end{proof}

\section{Quasi-Newton Method}\label{section:quasi-Newton}
To solve the reduced minimization problem~\eqref{eq: def_varphi} corresponding to the RSNLS problem~\eqref{eq: nlls_regy}, we apply a quasi-Newton method whose iterations are defined by 
\begin{equation*}
\by^{(k+1)} = \by^{(k)} + \bs^{(k)}, \,k = 0,1,2,...,
\end{equation*}
where $\bs^{(k)}$ is the solution of the system of equations
\begin{equation*}
\bH(\by^{(k)})\bs^{(k)} =  -\nabla \varphi(\by^{(k)}).
\end{equation*}
The right-hand side of the system contains the negative of the gradient of $\varphi(\by)$, which is defined by
\[ \nabla \varphi(\by) = \bJ^\top_{\mathbf{f}}(\by) \mathbf{f}(\by)+ \nabla\mathcal{R}(\by),\]
where $\bJ_{\mathbf{f}}(\by)$ is the Jacobian of $\mathbf{f}$, whose columns are 
$\dfrac{\partial \bff(\by)}{\partial y_j}$ (see \eqref{eq: Jac_columns}). The matrix of the system is an approximation of the Hessian $\nabla^2 \varphi(\by)$ defined as
\begin{equation}\label{eq: approxHessian}
 \bH(\by): =\bJ^\top_{\mathbf{f}} (\by) \bJ_{\mathbf{f}}(\by) + \nabla^2 \mathcal{R}(\by).
\end{equation}
Algorithm~\ref{Alg:RGenVarPro} summarizes our approach.

\begin{algorithm}[ht]
\caption{\texttt{RGenVarPro}}
\label{Alg:RGenVarPro}
\begin{algorithmic}[1]
\STATE {\bf Input:} A map $\by \mapsto \bA(\by)$, $\bL$, $\bb$, $\by^{(0)}$, $\lambda$, and $\mathcal{R}(\by)$ 
\FOR{$k = 0, 1, \dots$ until a stopping criterion is satisfied}
    \STATE $\bx^{(k)} = \left(\bA(\by^{(k)})^\top\bA(\by^{(k)})+\lambda^2 \bL^\top \bL\right)^{-1}\bA(\by^{(k)})^\top\bb$
    \STATE $\bff^{(k)} = \left[\begin{array}{c} \bA(\by^{(k)}) \\ \lambda \bL\end{array}\right] \bx^{(k)}-\left[\begin{array}{c} \bb \\ \mathbf{0} \end{array}\right]$\;
    \STATE Compute the Jacobian $\bJ_{\bff}(\by^{(k)})$ and the approximate Hessian $\bH(\by^{(k)})$\;
    \STATE $\bs^{(k)}=  -\bH(\by^{(k)})^{-1}\Big(\bJ^\top_{\mathbf{f}}(\by^{(k)}) \mathbf{f}^{(k)}+ \nabla\mathcal{R}(\by^{(k)})\Big)$\;
    \STATE $\by^{(k+1)} = \by^{(k)} + \bs^{(k)}$\;
\ENDFOR    
\end{algorithmic}
\end{algorithm}

\subsection{Local Convergence} Now we show the local convergence of Algorithm~\ref{Alg:RGenVarPro}. We adapt the results in \cite{espanol2025local} for this setting of separable nonlinear least squares problems of the form \eqref{eq: nlls_regy} with $\calR(\by)$ being convex and twice differentiable, and the linear variable $\bx$ being solved exactly.

Let $\by^*$ be a solution of the reduced problem \eqref{eq: minvarphi} and $\{\by^{(k)}\}_{k \geq 0}$ the iterates generated by Algorithm~\ref{Alg:RGenVarPro}. For $\varepsilon>0$, let $B(\by^*, \varepsilon)$ denote the open ball of radius $\varepsilon$ centered at $\by^*$. The following lemmas are needed for the main local convergence result that follows.
\begin{lemma}
Assume that the Jacobian $\bJ_{\bf f}(\by)$ is Lipschitz continuous in $B(\by^*, \varepsilon)$, i.e., there exists $L_\bff > 0$ such that
\begin{equation*}
\|\bJ_{\bf f}(\by_1) - \bJ_{\bf f}(\by_2)\| \leq L_\bff \|\by_1 - \by_2\|, \quad \mbox{for all } \,\by_1, \by_2 \in B(\by^*, \varepsilon). 
\end{equation*}
Also, suppose $\calR(\by)$ is $C^2$ and $\nabla^2 \calR(\by)$ is Lipschitz continuous with constant $L_\calR$. Then, the approximate Hessian $\bH(\by)$ introduced in \eqref{eq: approxHessian} is also Lipschitz continuous on $B(\by^*, \varepsilon)$ with constant $2 K L_\bff + L_\calR$, where $K := \sup_{\by \in \overline{B(\by^*, \varepsilon)}} \|\bJ_{\bf f}(\by)\|$.
\end{lemma}

\begin{proof} Let $\by_1, \by_2 \in B(\by^*, \varepsilon)$. Then, we have that
\begin{align}\label{ineq: jac1}
\|\bH(\by_1) - \bH(\by_2)\|&=\|\bJ^\top_{\bf f}(\by_1) \bJ_{\bf f}(\by_1) + \nabla^2 \calR(\by_1) - \bJ^\top_{\bf f}(\by_2) \bJ_{\bf f}(\by_2) - \nabla^2 \calR(\by_2)\| \nonumber \\
&\leq \|\bJ^\top_{\bf f}(\by_1) \bJ_{\bf f}(\by_1) - \bJ^\top_{\bf f}(\by_2) \bJ_{\bf f}(\by_2)\| + \|\nabla^2 \calR(\by_1) - \nabla^2 \calR(\by_2)\|.
\end{align}
We now bound from above the first term of \eqref{ineq: jac1}:
\begin{align*} 
\|\bJ^\top_{\bf f}(\by_1) \bJ_{\bf f}(\by_1) - &\bJ^\top_{\bf f}(\by_2) \bJ_{\bf f}(\by_2)\|\\
&= \|\bJ^\top_{\bf f}(\by_1)(\bJ_{\bf f}(\by_1) - \bJ_{\bf f}(\by_2)) + (\bJ^\top_{\bf f}(\by_1) - \bJ^\top_{\bf f}(\by_2)) \bJ_{\bf f}(\by_2)\| \\
&\leq \|\bJ^\top_{\bf f}(\by_1)\| \|\bJ_{\bf f}(\by_1) - \bJ_{\bf f}(\by_2)\| + \|\bJ_{\bf f}(\by_2)\| \|\bJ^\top_{\bf f}(\by_1) - \bJ^\top_{\bf f}(\by_2)\| \\
&\leq 2 K L_\bff \|\by_1 - \by_2\|.
\end{align*}
The second term of \eqref{ineq: jac1} is bounded by $L_\calR \|\by_1 - \by_2\|$. Hence, the result follows.
\end{proof}

\begin{lemma}\label{lem: Hinverse}
Assume the same conditions as in the previous lemma, that is, that the Jacobian $\bJ_{\bf f}(\by)$ is Lipschitz continuous in $B(\by^*, \varepsilon)$, $\calR(\by)$ is $C^2$, and $\nabla^2 \calR(\by)$ is Lipschitz with constant $L_\calR$. In addition, assume that $\bH(\by^*)$ is invertible. Then, there exists $0 < \delta \leq \varepsilon$ such that for all $\by \in B(\by^*, \delta)$, we have that
$\|\bH(\by)\| \leq 2 \|\bH(\by^*)\|$,
$\bH(\by)$ is invertible, and $\|\bH(\by)^{-1}\| \leq 2 \|\bH(\by^*)^{-1}\|.$
\end{lemma}
\begin{proof}
By the previous lemma, $\bH(\by)$ is Lipschitz continuous, so we can write
\begin{equation*}
\|\bH(\by)\| - \|\bH(\by^*)\| \leq \|\bH(\by) - \bH(\by^*)\| \leq L_H \|\by - \by^*\|
\end{equation*}
for some constant $L_H$. Choosing $\delta_1 =\min\{\|\bH(\by^*)\|/L_H, \varepsilon\}$, for all $\by \in B(\by^*, \delta_1)$, we obtain $\|\bH(\by)\| \leq 2 \|\bH(\by^*)\|$. 

To establish the second inequality, we write
\begin{equation*}
\|\bI - \bH(\by^*)^{-1} \bH(\by)\| = \|\bH(\by^*)^{-1} (\bH(\by^*) - \bH(\by))\| \leq L_H \|\bH(\by^*)^{-1}\|\|\by - \by^*\|,
\end{equation*}
and choose
$\delta_2 = \min\{ 1/ (2 L_H \|\bH(\by^*)^{-1}\|), \varepsilon\}.$
Then, for any $\by \in B(\by^*, \delta_2)$, we have that
\begin{equation*}
\|\bI - \bH(\by^*)^{-1} \bH(\by)\| \leq \frac{1}{2}.
\end{equation*}
and therefore, we can apply Banach’s Lemma~\cite[Theorem 1.2.1]{kelley1995iterative}  to the matrices $\bH(\by^*)^{-1}$ and $\bH(\by)$, and conclude that, for all  $\by \in B(\by^*, \delta_2)$, we have that $\bH(\by)$ is invertible and
$\|\bH(\by)^{-1}\| \leq 2 \|\bH(\by^*)^{-1}\|.$ The lemma follows taking $\delta = \min\{\delta_1, \delta_2\}$.
\end{proof}

\begin{theorem}\label{thm: error bound}
Let $\by^*$ be a local minimizer of $\varphi(\by) = \frac{1}{2} \|\bff(\by)\|^2 + \calR(\by)$. Assume that $\bJ_{\bf f}(\by)$ is Lipschitz continuous and has full column rank in a neighborhood $B(\by^*, \varepsilon)$, and that $\calR$ is convex, $C^2$,  and has a Lipschitz continuous Hessian. Define the error $\be(\by) = \by - \by^*$ and let ${\be}_k = \be(\by^{(k)})$ denote the error at the $k$th iteration. If $ \|\bH(\by^*)^{-1}\|\|\nabla^2\varphi(\by^*)-\bH(\by^*)\| < 1/2$, then there exist constants $K > 0$ and $\delta > 0$ such that, for all $\by \in B(\by^*, \delta)$, the iterates $\by^{(k)}$ produced by Algorithm \texttt{RGenVarPro} with initial vector $\by^{(0)}=\by$ satisfy
\begin{equation*}
\|\be_{k+1}\| \leq \left(K \|\be_{k}\| + 2 \|\bH(\by^*)^{-1}\|\|\nabla^2 \varphi(\by^*) - \bH(\by^*)\| \right) \|\be_{k}\|.
\end{equation*}
\end{theorem}

\begin{proof} 
By hypothesis, $\bJ_{\bf f}(\by^*)$ has full column rank and $\mathcal{R}(\by)$ is convex. Consequently, the Hessian approximation $\bH(\by^*) =\bJ^\top_{\mathbf{f}} (\by^*) \bJ_{\mathbf{f}}(\by^*) + \nabla^2 \mathcal{R}(\by^*)$
is strictly positive definite and hence invertible. By Lemma \ref{lem: Hinverse}, there exists $0 < \delta \le \varepsilon$ such that $\bH(\by)$ is invertible for all $\by \in B(\by^*,\delta)$. The iterates of the algorithm are then defined, for any initial vector $\by^{(0)} \in B(\by^*,\delta)$, by
\begin{equation}\label{eq: iteration}
\by^{(k+1)} = \by^{(k)} - 
\bH(\by^{(k)})^{-1}
\nabla \varphi(\by^{(k)})
\end{equation}
provided that $\by^{(k)}\in B(\by^*,\delta)$, i.e., $\|\be_k\|<\delta $, for all $k$.

Subtracting $\by^*$ from both sides in \eqref{eq: iteration} and rearranging terms, we obtain
\begin{small}
\begin{align*}
    \be_{k+1} &= \be_k - \bH(\by^{(k)})^{-1} \left( \bJ_{\bf f}(\by^{(k)})^\top \bff(\by^{(k)}) + \nabla \mathcal{R}(\by^{(k)}) \right) \\
    &= \bH(\by^{(k)})^{-1} \left( \bH(\by^{(k)})\be_k - \bJ_{\bf f}(\by^{(k)})^\top \bff(\by^{(k)}) - \nabla \mathcal{R}(\by^{(k)}) \right) \\
    &= \bH(\by^{(k)})^{-1} \left( (\bJ_{\bf f}(\by^{(k)})^\top \bJ_{\bf f}(\by^{(k)}) + \nabla^2 \mathcal{R}(\by^{(k)}))\be_k - \bJ_{\bf f}(\by^{(k)})^\top \bff(\by^{(k)}) - \nabla \mathcal{R}(\by^{(k)}) \right) \\
    &= \bH(\by^{(k)})^{-1} \bigg(
    \underbrace{\bJ_{\bf f}(\by^{(k)})^\top \left( \bJ_{\bf f}(\by^{(k)})\be_k - \bff(\by^{(k)}) \right)}_{A}
    + \underbrace{ \nabla^2 \mathcal{R}(\by^{(k)})\be_k - \nabla \mathcal{R}(\by^{(k)}) }_{B} \bigg).
\end{align*}
\end{small}

Since $\by^*$ is a stationary point, we have that
$\bJ_{\bf f}(\by^*)^\top \bff(\by^*) + \nabla \mathcal{R}(\by^*) = 0$, and
so
\begin{align*}
A+ B &= \left(A+\bJ_{\bf f}(\by^*)^\top \bff(\by^*) \right)+ \left(B+\nabla \mathcal{R}(\by^*) \right).
\end{align*}

We first rewrite $A+\bJ_{\bf f}(\by^*)^\top \bff(\by^*)$ as follows:

\begin{align}\label{eq: A}
A +\bJ_{\bf f}(\by^*)^\top \bff(\by^*) &= \bJ_{\bf f}(\by^{(k)})^\top \left( \bJ_{\bf f}(\by^{(k)})\be_k - \bff(\by^{(k)})+ \bff(\by^*)\right) \notag \\
& \qquad {} - \left(\bJ_{\bf f}(\by^{(k)})^\top -\bJ_{\bf f}(\by^*)^\top \right) \bff(\by^*). 
\end{align}
Note that
\begin{equation*}
\bJ_{\bf f}(\by^{(k)})\be_k - \bff(\by^{(k)}) +\bff(\by^*) 
=   \left(\bJ_{\bf f}(\by^{(k)})- \bJ_{\bf f}(\by^*)\right)\be_k + \left(\bJ_{\bf f}(\by^*)\be_k+\bff(\by^*) -\bff(\by^{(k)})\right) .
\end{equation*}
By Taylor's theorem applied to $\bff$,
\[\|\bJ_{\bf f}(\by^*)\be_k+\bff(\by^*) -\bff(\by^{(k)}) \| \le T_\bff \| \be_k \|^2 \]
for a constant $T_\bff$.

Thus,
\begin{align}\label{eq: bound A1}
\qquad &\left\| \bJ_{\bf f}(\by^{(k)})\be_k  -  \bff(\by^{(k)}) + \bff(\by^*) \right\| \notag \\
& \qquad \leq \| \bJ_{\bf f}(\by^{(k)}) - \bJ_{\bf f}(\by^*) \| \|\be_k\| +  \|\bJ_{\bf f}(\by^*)\be_k+\bff(\by^*) -\bff(\by^{(k)}) \|
\leq C_1 \|\be_k\|^2,
\end{align}
where $C_1 = L_\bff + T_\bff$, with $L_\bff$ being the Lipschitz constant of $\bJ_{\bf f}$.

Using the Taylor expansion $\bJ_{\bf f}(\by^*) + \bJ'_{\bf f}(\by^*)\be_k$  of $\bJ_{\bf f}$ around $\by^*$, where \(\bJ'_{\bf f}(\by^*)\) denotes the derivative of the Jacobian (i.e., the third-order tensor whose action on a direction vector gives the directional derivative of \(\bJ_{\bf f}\) at \(\by^*\)),
and the identity
\[
\nabla^2\varphi(\by^*)
= \bJ_{\bf f}(\by^*)^\top \bJ_{\bf f}(\by^*) + \bJ'_{\bf f}(\by^*)^\top\bff(\by^*) + \nabla^2 \mathcal{R}(\by^*) = \bH(\by^*) + \bJ'_{\bf f}(\by^*)^\top\bff(\by^*),
\]
we obtain
\begin{align*}
\left(\bJ_{\bf f}(\by^{(k)}) - \bJ_{\bf f}(\by^*) \right)^\top\bff(\by^*) &=  \be_k^\top \bJ'_{\bf f}(\by^*)^\top \bff(\by^*)  \\ & \qquad\qquad {} + \left(\bJ_{\bf f}(\by^{(k)}) - \bJ_{\bf f}(\by^*) - \bJ'_{\bf f}(\by^*)\be_k\right)^\top\bff(\by^*) \\
&= \be_k^\top \left(\nabla^2\varphi(\by^*)-\bH(\by^*)\right)  \\ & \qquad\qquad {} +\left(\bJ_{\bf f}(\by^{(k)}) - \bJ_{\bf f}(\by^*) - \bJ'_{\bf f}(\by^*)\be_k\right)^\top\bff(\by^*).
\end{align*}
Then, $\|\bJ_{\bf f}(\by^{(k)}) - \bJ_{\bf f}(\by^*) - \bJ'_{\bf f}(\by^*)\be_k\|\le T_{\bJ_\bff} \|\be_k\|^2$ for a constant $T_{\bJ_\bff}$ and, therefore,
\begin{align}\label{eq: bound A2}
    & \left\|(\bJ_{\bf f}(\by^{(k)}) -  \bJ_{\bf f}(\by^*))^\top \bff(\by^*)\right\| \nonumber\\
    & \quad\qquad \le \|\nabla^2\varphi(\by^*)-\bH(\by^*)\| \|\be_k\|+ \|\bJ_{\bf f}(\by^{(k)}) - \bJ_{\bf f}(\by^*) - \bJ'_{\bf f}(\by^*)\be_k\| \|\bff(\by^*)\|   \nonumber \\
    & \quad \qquad \le \|\nabla^2\varphi(\by^*)-\bH(\by^*)\| \|\be_k\|+  T_{\bJ_\bff} \|\bff(\by^*)\| \|\be_k\|^2.   
\end{align}

From identity \eqref{eq: A}, and the bounds in \eqref{eq: bound A1} and \eqref{eq: bound A2} we obtain
\begin{align}\label{eq: bound A}
    \| A+ \bJ_{\bf f}(\by^*)^\top \bff(\by^*)\| 
    & \le  C_1 \| \bJ_{\bf f}(\by^{(k)})\| \|\be_k\|^2 +\|\nabla^2\varphi(\by^*)-\bH(\by^*)\| \|\be_k\| \nonumber \\ & \qquad \qquad \qquad\qquad \qquad \qquad \qquad \qquad \qquad{} + T_{\bJ_\bff} \|\bff(\by^*)\| \|\be_k\|^2 \nonumber\\ 
    & \le (C_1 \| \bJ_{\bf f}(\by^{(k)})\| + C_2) \|\be_k\|^2 +\|\nabla^2\varphi(\by^*)-\bH(\by^*)\| \|\be_k\|
\end{align}
for a constant $C_2$.

Next, we consider the term $B+ \nabla \mathcal{R}(\by^{*})$. Using the Taylor expansion of $\nabla \mathcal{R}$ around $\by^*$, we have that
\begin{align*}
B + \nabla \mathcal{R}(\by^*)  & = \nabla^2 \mathcal{R}(\by^{(k)})\be_k - \nabla \mathcal{R}(\by^{(k)}) +\nabla \mathcal{R}(\by^*)  \\
&=  \nabla^2 \mathcal{R}(\by^{(k)}) \be_k   - \nabla^2 \mathcal{R}(\by^*)\be_k +\left(\nabla \mathcal{R}(\by^*) + \nabla^2 \mathcal{R}(\by^*)\be_k - \nabla \mathcal{R}(\by^{(k)}) \right) \\
&=  (\nabla^2 \mathcal{R}(\by^{(k)}) -\nabla^2 \mathcal{R}(\by^*))\be_k +\left(\nabla \mathcal{R}(\by^*) + \nabla^2 \mathcal{R}(\by^*)\be_k - \nabla \mathcal{R}(\by^{(k)}) \right).
\end{align*}
By the Lipschitz assumption on the Hessian of $\mathcal{R}$, we have
\begin{equation}\label{eq: bound B1}
\| \nabla^2 \mathcal{R}(\by^{(k)}) -\nabla^2 \mathcal{R}(\by^*)\| \le L_{\cal R} \| \be_k\|,
\end{equation}
with $L_{\cal R}$ the Lipschitz constant. In addition, 
\begin{equation}\label{eq: bound B2}
\|\nabla \mathcal{R}(\by^*) + \nabla^2 \mathcal{R}(\by^*)\be_k - \nabla \mathcal{R}(\by^{(k)}) \|\le T_{\cal R} \| \be_k\|^2
\end{equation}
for a constant $T_{\cal R}$. From inequalities \eqref{eq: bound B1} and \eqref{eq: bound B2}, it follows that
\begin{equation}\label{eq: bound B}
\| B+ \nabla \mathcal{R}(\by^*)\| \le (L_{\cal R}+ T_{\cal R}) \| \be_k\|^2.
\end{equation}

Now, combining the bounds in \eqref{eq: bound A} and \eqref{eq: bound B}, we obtain
\begin{align*}
\| A+B\| &\le 
\|A+\bJ_{\bf f}(\by^*)^\top \bff(\by^*) \|+ \|B+\nabla \mathcal{R}(\by^*) \| \\
&\le (C_1 \| \bJ_{\bf f}(\by^{(k)})\| + C_2)\|\be_{k}\|^2 + \|\nabla^2 \varphi(\by^*) - \bH(\by^*)\| \|\be_{k}\| + (L_{\cal R}+ T_{\cal R}) \| \be_k\|^2 \\
&=(C_1 \| \bJ_{\bf f}(\by^{(k)})\| + C_3 )\|\be_{k}\|^2 +  \|\nabla^2 \varphi(\by^*) - \bH(\by^*)\| \|\be_{k}\|
\end{align*}
for a constant $C_3$.

From the recurrence of the error $\be_{k+1} = \bH(\by^{(k)})^{-1} (A+B)$, we deduce that
\begin{align*}
\|\be_{k+1}\| & \leq \|\bH(\by^{(k)})^{-1}\| \| A + B\| \\
& \leq 2 \|\bH(\by^*)^{-1}\|
\left( \left( C_1 \|\bJ_{\bf f}(\by^{(k)})\| + C_3 \right) \|\be_k\|^2
+ \left\| \nabla^2 \varphi(\by^*) - \bH(\by^*) \right\| \|\be_k\| \right) \\
& \leq 2 \|\bH(\by^*)^{-1}\|
\left( C \|\be_k\|^2
+ \left\| \nabla^2 \varphi(\by^*) - \bH(\by^*) \right\| \|\be_k\| \right) 
\end{align*}
for a suitable constant $C$, since $ \|\bJ_{\bf f}(\by)\|$ is bounded in $\overline{B(\by^*, \varepsilon)}$ by the continuity assumption on $\bJ_\bff$.
Choosing $K = 2 \|\bH(\by^*)^{-1}\| C$, we conclude that
\begin{equation*}
\|\be_{k+1}\| \leq  K \|\be_k\|^2 +  2 \|\bH(\by^*)^{-1}\|\left\| \nabla^2 \varphi(\by^*) - \bH(\by^*) \right\|  \|\be_k\| .
\end{equation*}

Finally, from the inequality 
\[ \|\be_{k+1}\|\leq 2 \|\bH(\by^*)^{-1}\|
\left( C \|\be_k\|
+ \left\| \nabla^2 \varphi(\by^*) - \bH(\by^*) \right\| \right) \|\be_k\|, \]
since  $\|\bH(\by^*)^{-1}\| \left\| \nabla^2 \varphi(\by^*) - \bH(\by^*) \right\| < \frac{1}{2} $,
by taking $\delta >0$ sufficiently small, we have that whenever $\|\be_k\|<\delta$, it follows that $\|\be_{k+1}\|\le \|\be_k\|$. In particular, this implies that $\by^{(k+1)}\in B(\by^*, \delta)$.
\end{proof}

For problems with low nonlinearity, that is, problems that are technically nonlinear but exhibit nearly linear behavior in the region of interest, the difference between the Hessian and its approximation is often negligible. In such cases, Algorithm~\ref{Alg:RGenVarPro} can achieve superlinear convergence rates, and even quadratic convergence under certain conditions. We formalize this in the following corollary.

\begin{corollary}
Under the same assumptions and notation as in Theorem \ref{thm: error bound}, define
$
\eta := 2 \|\bH(\by^*)^{-1}\|\|\nabla^2 \varphi(\by^*) - \bH(\by^*)\|.
$
If $0 \leq \eta < 1$ and $\|\be_0\| < \min\{ \delta, \frac{1 - \eta}{2K}\},$
then, for all $k\ge 0$,
\[
\|\be_k\| < \left( \frac{1 + \eta}{2} \right)^k \|\be_0\|,
\]
guaranteeing linear convergence. In particular, if $\nabla^2 \varphi(\by^*) = \bH(\by^*)$, then we have quadratic convergence.
\end{corollary}

\begin{proof}
By Theorem \ref{thm: error bound}, taking $\by^{(0)} \in B(\by^*, \delta)$, for all $k\ge 0$, we have
\[
\|\be_{k+1}\| \leq  \left( K \|\be_k\| + \eta \right) \|\be_k\|. 
\]
For $k = 0$, the assumption
$
 \|\be_0\| < \frac{1 - \eta}{2K}$ implies that
$K \|\be_0\| + \eta < \frac{1 + \eta}{2} < 1,$
so
\[
\|\be_1\| \leq \left( K \|\be_0\| + \eta \right) \|\be_0\| < \frac{1 + \eta}{2} \|\be_0\|.
\]
Assuming inductively that $\|\be_k\| < \left( \frac{1 + \eta}{2} \right)^k \|\be_0\|$, we have
\begin{align*}
\|\be_{k+1}\| \leq \left( K \|\be_k\| + \eta \right) \|\be_k\| 
& < \left( K \|\be_0\| + \eta \right) \left( \frac{1 + \eta}{2} \right)^k \|\be_0\| 
< \left( \frac{1 + \eta}{2} \right)^{k+1} \|\be_0\|.
\end{align*}
The first claim follows by induction.

If $\eta = 0$, then the recurrence in Theorem \ref{thm: error bound} simplifies to
$
\|\be_{k+1}\| \leq K \|\be_k\|^2,
$
which implies quadratic convergence.
\end{proof}

\section{Inexact RGenVarPro}\label{section:iRGenVarPro}
We extend the Inexact GenVarPro method introduced in \cite{espanol2025local} to address large-scale separable nonlinear least squares problems of the form \eqref{eq: nlls_regy}, now including a general regularization term $\mathcal{R}(\by)$. This extension reduces the computational cost of Algorithm~\ref{Alg:RGenVarPro} by approximating $\bx(\by)$ in Step~3 via an iterative solver.

Specifically, we adopt an LSQR-based approach for solving the linear subproblem, terminating the iterations when the following relative residual condition is met:
\begin{equation}\label{eq: stopping_criterion}
\frac{\|\bK(\by)^\top \br^{(i)}\|}{\|\br^{(i)}\| \|\bK(\by)\|} < \varepsilon,
\end{equation}
for a prescribed tolerance $\varepsilon > 0$, where $\br^{(i)}$ denotes the residual at the $i$-th LSQR iteration. This approach follows the strategy outlined in \cite{espanol2025local}.

At each outer iteration $k$, we construct an approximate Jacobian matrix $\bar{\bJ}^{(k)}$, whose $j$-th column is given by
\begin{small}
\begin{equation}\label{eq: Jappkj}
 [\bar\bJ^{(k)}]_j = \mathcal{P}^\perp_{\bK(\by^{(k)})} \left[ \begin{array}{c} \frac{\partial \bA}{\partial y_j}(\by^{(k)}) \\ \bf{0} \end{array} \right] \bar{\bx}^{(k)} + \left(\bK^\dagger(\by^{(k)})\right)^\top \left(\frac{\partial \bA}{\partial y_j}(\by^{(k)})\right)^\top \left(\bb - \bA(\by^{(k)}) \bar{\bx}^{(k)}\right),
\end{equation}
\end{small}
for $j = 1, \dots, r$, where $\bar{\bx}^{(k)}$ denotes the approximate solution to the regularized linear least squares problem
\[
\min_{\bx} \frac{1}{2}\|\bK(\by^{(k)})\bx - \bd\|^2.
\]
The LSQR solver is applied to this subproblem with a tolerance $\varepsilon^{(k)} > 0$, which is updated adaptively from an initial value $\varepsilon^{(0)}$. Using the approximate Jacobian $\bar{\bJ}^{(k)}$, we define a corresponding approximate Hessian at the $k$th iteration as
\begin{equation}\label{eq: defbarH}
\bar{\bH}^{(k)} := (\bar{\bJ}^{(k)})^\top \bar{\bJ}^{(k)} + \nabla^2 \mathcal{R}(\by^{(k)}).
\end{equation}
Additionally, the residual at iteration $k$ is defined as
\[
\bg^{(k)} = \bK(\by^{(k)}) \bar{\bx}^{(k)} - \bd.
\]
Then, the \texttt{iRGenVarPro} iteration is defined by 
\begin{equation}\label{eq: iGenVarPro}
\by^{(k+1)}=\by^{(k)} + \bt^{(k)} \mbox{ with } \bt^{(k)}= -(\bar{\bH}^{(k)})^{-1} \Big((\bar{\bJ}^{(k)})^\top \bg^{(k)}+ \nabla\mathcal{R}(\by^{(k)})\Big).
\end{equation}
Algorithm~\ref{Alg:I-GenVarPro} summarizes the \texttt{iRGenVarPro} method applied to problem \eqref{eq: nlls_regy}.

\begin{algorithm}[ht]
\caption{\texttt{iRGenVarPro}}
\label{Alg:I-GenVarPro}
\begin{algorithmic}[1]
\STATE {\bf Input:} A map $\by \mapsto \bA(\by)$, $\bL$, $\bb$, $\by^{(0)}$, and $\varepsilon^{(0)}>0$
\FOR{$k = 0, 1, \dots$ until a stopping criterion is satisfied}
    \STATE Compute $\bar \bx^{(k)}$ by applying the LSQR algorithm with stopping criterion \eqref{eq: stopping_criterion} for $\varepsilon=\varepsilon^{(k)}$ to the problem 
    $\min_{\bx} \frac{1}{2}\|\bK(\by^{(k)})\bx - \bd\|^2$ 
    \STATE $\bg^{(k)} = \bK(\by^{(k)})\bar \bx^{(k)}-\bd$\;
    \STATE Compute the approximate Jacobian matrix $\bar\bJ^{(k)}$ according to \eqref{eq: Jappkj}\;
    \STATE Compute the approximate Hessian $\bar{\bH}^{(k)}$ according to \eqref{eq: defbarH}\;
    \STATE $\bt^{(k)}=  -(\bar{\bH}^{(k)})^{-1} \Big((\bar{\bJ}^{(k)})^\top \bg^{(k)}+ \nabla\mathcal{R}(\by^{(k)})\Big)$\;
    \STATE $\by^{(k+1)} = \by^{(k)} + \bt^{(k)}$\;
    \STATE $\varepsilon^{(k+1)} = \varepsilon^{(k)}/2$\;
\ENDFOR    
\end{algorithmic}
\end{algorithm}

\subsection{Local Convergence} We now prove the local convergence of \texttt{iRGenVarPro}. First, we restate the following lemma, which follows from~\cite[Lemma 3.1]{espanol2025local}.

\begin{lemma}\label{lemma:approx_residual}
For a fixed $\by$, let $\bx$ be the (exact) solution of $\min_{\bx} \frac{1}{2}\|\bK(\by)\bx - \bd\|^2$ and $\bar \bx$ the approximate solution computed by the LSQR algorithm with stopping criterion
\begin{equation*}
\frac{\|\bK(\by)^\top \br^{(i)}\|}{\|\br^{(i)}\|\|\bK(\by)\|}<\varepsilon,
\end{equation*}
for a sufficiently small prescribed tolerance $\varepsilon>0$, with $\br^{(i)}$ being the residual at the $i$-th iteration. If $\br=\bK(\by)\bx- \bd$ and $ \bar \br = \bK(\by)\bar\bx-\bd$ are the corresponding residuals, then
\begin{equation*}
\|\br - \bar \br\| <   \frac{ 2 \kappa_2(\bK(\by))}{1-\varepsilon\, \kappa_2(\bK(\by))}\|\bb\|\varepsilon.
\end{equation*}
\end{lemma}

Now, we state the main result corresponding to the local convergence of \linebreak \texttt{iRGenVarPro}.
\begin{theorem}\label{thm: main} Let $\by^*$ be a minimizer of $\varphi(\by)$ and let $\Omega\subseteq \mathbb{R}^r$ be an open neighborhood  of $\by^*$. 

Assume that:
\begin{itemize}
\item there are constants $M>0$,  $\kappa>0$, and $\gamma >0$ such that $\max_j \left\|\frac{\partial \bA (\by)}{\partial y_j}\right\| \le M$, $\kappa_2(\bK(\by))\le \kappa$, and $\| \bK(\by)\| \ge \gamma$ for all $\by$ in $\Omega$, where $\bK(\by)$ is the matrix defined in \eqref{eq: defKd},
\item the Jacobian $\bJ= \bJ_{\bff}(\by)$ of $\bff(\by)$ is of full rank in $\Omega$ and satisfies the Lipschitz condition $\|\bJ(\by) - \bJ(\by^*) \|\le L \| \by - \by^\ast\|$ with a constant $L>0$ for all $\by$ in $\Omega$,
\item the regularization term $\calR$ is convex and $C^2$, and its Hessian satisfies the Lipschitz condition $\|\nabla^2 \mathcal{R}(\by) - \nabla^2 \mathcal{R}(\by^*) \|\le L_\mathcal{R} \| \by - \by^\ast\|$ for all $\by$ in $\Omega$.
\end{itemize}
Then, if $ \|\bH(\by^*)^{-1}\| \| \bJ(\by^*)^\top \bJ(\by^*) -\nabla^2\phi(\by^*)\| < 1/32$,  where $\phi(\by)=\frac{1}{2}\|\bff(\by)\|^2$, and $\|\by^{(0)}-\by^*\|$ is sufficiently small, one can choose adequate tolerances $\varepsilon^{(k)}>0$ for $k \ge 0$ so that the sequence $(\by^{(k)})_{k\ge 0}$ generated by the  \texttt{iRGenVarPro} method satisfies  $\|\by^{(k)}-\by^*\|\le \dfrac{1}{2^k}$ for every $k\in \mathbb{N}$.
\end{theorem}

\begin{proof}
This proof closely follows that of \cite[Theorem 3.3]{espanol2025local}. For $k\ge 0$, let $\be_k = \by^{(k)}- \by^*$. We will prove the stated upper bound on $\|\be_k\|$ by induction on $k$.

For $k\ge 0$, let $\bJ^{(k)}$ be the Jacobian at $\by^{(k)}$ and $\bar \bJ^{(k)}$ its approximation defined in \eqref{eq: Jappkj}. We also write
$$\bff^{(k)} = \bff(\by^{(k)}) = \bK(\by^{(k)}) \bx^{(k)} - \bd 
\quad \mbox{and} \quad 
\bg^{(k)} = \bK(\by^{(k)}) \bar \bx^{(k)} -\bd,$$
where $\bx^{(k)}=\bx(\by^{(k)})$ is the exact solution to  $\min_{\bx} \frac{1}{2}\|\bK(\by^{(k)})\bx - \bd\|^2$.
Let $\bJ^\ast=\bJ(\by^\ast)$ and $\bH^\ast=\bH(\by^*)$. Assume $\alpha>0$ is an upper bound for $\|\bJ^\ast\|$ and $\| \bar\bJ^{(k)}\|$ for all $k$. 
Set $\beta= \|(\bH^\ast)^{-1}\|$ and $\delta^\ast= \| (\bJ^\ast)^\top \bJ^\ast- \nabla^2 \phi (\by^\ast) \|$.

From the definition of the quasi-Newton iteration \eqref{eq: iGenVarPro} and using that $\by^*$ is a minimizer of $\varphi$ and, consequently,  $\nabla \varphi (\by^\ast)=0$, we can write
\begin{align}\label{eq:error_expression}
\be_{k+1} &= (\bar\bH^{(k)})^{-1}\Bigg( 
    \left[ -\nabla\varphi(\by^{(k)}) + \nabla \varphi (\by^\ast) + \nabla^2 \varphi(\by^\ast) \be_{k} \right] \nonumber\\
&\qquad + \left[ \nabla \varphi(\by^{(k)}) - (\bar\bJ^{(k)})^\top \bg^{(k)} - \nabla\mathcal{R}(\by^{(k)}) \right] + \left[\bar\bH^{(k)} - \nabla^2 \varphi(\by^\ast) \right] \be_{k} \Bigg) \nonumber\\
&=  (\bar\bH^{(k)})^{-1}\Bigg( \left[ -\nabla\varphi(\by^{(k)}) + \nabla \varphi (\by^\ast) + \nabla^2 \varphi(\by^\ast) \be_{k} \right] \nonumber\\
&\qquad + \left[ (\bJ^{(k)})^\top \bff^{(k)} - (\bar\bJ^{(k)})^\top \bg^{(k)} \right] + \left[(\bar\bJ^{(k)})^\top\bar\bJ^{(k)} - \nabla^2 \phi(\by^\ast) \right] \be_{k} \Bigg).
\end{align}

We will now bound the norm of each of the three terms in the above expression in terms of $\varepsilon^{(k)}$ and $\| \be_{k}\|$.
Regarding the first term, we have that by Taylor expansion and the smoothness of $\varphi$, there is a constant $T>0$ such that
\begin{equation}\label{eq:taylor}
\|-\nabla \varphi (\by^{(k)}) + \nabla \varphi (\by^\ast) + \nabla^2 \varphi(\by^\ast) \be_{k}\| \le T \| \be_{k}\|^2.
\end{equation}

In order to get an upper  bound for the second term in \eqref{eq:error_expression}, we write
\begin{align*}
\| (\bJ^{(k)})^\top \bff^{(k)} - (\bar\bJ^{(k)})^\top \bg^{(k)}\|\le\| (\bJ^{(k)})^\top  - (\bar\bJ^{(k)})^\top  \| \|\bff^{(k)}\|+  \|(\bar\bJ^{(k)})^\top \| \|\bff^{(k)}- \bg^{(k)}\|.
\end{align*}
Assume that $\varepsilon^{(k)} \kappa <\frac{1}{2}$. By \cite[Lemma 3.2]{espanol2025local}, 
\begin{align}\label{eq: approx_Jacobian}
\|\bJ^{(k)}  - \bar\bJ^{(k)}\| &<  8\sqrt{r(m+q)} M  \kappa_2^2(\bK(\by^{(k)})) \frac{\|\bb\|}{\|\bK(\by^{(k)})\|} \varepsilon^{(k)}
\le C  \varepsilon^{(k)},
\end{align}
where $C=8\sqrt{r(m+q)} M  \kappa^2 \gamma^{-1} \| \bb\|$. 
In addition, we have that 
$\|\bff^{(k)}\| \le \| \bb\|$, and Lemma \ref{lemma:approx_residual} implies:
 $$\| \bff^{(k)}- \bg^{(k)} \| <  4\kappa \|\bb\| \varepsilon^{(k)}. $$
Therefore, 
\begin{equation}\label{eq:gradient}
\|(\bJ^{(k)})^\top \bff^{(k)} - (\bar\bJ^{(k)})^\top \bg^{(k)}\|<  (C + 4\kappa \alpha )\| \bb\| \varepsilon^{(k)}.
\end{equation}

To get a bound for the last term in the expression \eqref{eq:error_expression} of $\be_{k+1}$, note that
$$\| (\bar\bJ^{(k)})^\top \bar\bJ^{(k)} - \nabla^2\phi (\by^\ast) \| \le \|(\bar\bJ^{(k)})^\top \bar\bJ^{(k)} - (\bJ^\ast)^\top \bJ^\ast\| +\|  (\bJ^\ast)^\top \bJ^\ast- \nabla^2 \phi (\by^\ast) \|.$$
Now, we have that
\begin{align*}
\|(\bar\bJ^{(k)})^\top \bar\bJ^{(k)} - (\bJ^\ast)^\top \bJ^\ast\|  &\le 
\|(\bar\bJ^{(k)})^\top - (\bJ^\ast)^\top \| \|\bar\bJ^{(k)} \|+ \|(\bJ^\ast)^\top \| \|\bar\bJ^{(k)} - \bJ^\ast \| \\ &\le 2\alpha \|\bar\bJ^{(k)} - \bJ^\ast \|
\end{align*}
and so,
\begin{equation}\label{eq:hessian}
\| ((\bar\bJ^{(k)})^\top \bar\bJ^{(k)} - \nabla^2 \phi (\by^\ast) ) \be_{k}\| \le (2\alpha \|\bar\bJ^{(k)} - \bJ^\ast \| +\delta^\ast) \|\be_{k}\|.
\end{equation}

Finally, we will obtain an upper bound for $\|(\bar\bH^{(k)})^{-1}\|$. To do so, we will need to show that
\[\|(\bH^\ast)^{-1} (\bar\bH^{(k)}-\bH^\ast)\|<1/2.\]

We have that
\begin{align*}
\|\bar\bH^{(k)}-\bH^\ast\|&=\|(\bar\bJ^{(k)})^\top \bar\bJ^{(k)} + \nabla^2 \mathcal{R}(\by^{(k)})- (\bJ^\ast)^\top \bJ^\ast - \nabla^2 \mathcal{R}(\by^*)\| \\
&\le \|(\bar\bJ^{(k)})^\top \bar\bJ^{(k)} - (\bJ^\ast)^\top \bJ^\ast \| + \| \nabla^2 \mathcal{R}(\by^{(k)}) - \nabla^2 \mathcal{R}(\by^*)\| \\
&\le 2\alpha  \|\bar\bJ^{(k)} - \bJ^\ast \| + L_\mathcal{R} \| \be_k\|. 
\end{align*}
By the bound in \eqref{eq: approx_Jacobian} and the Lipschitz assumption on $\bJ$, 
$$\|\bar\bJ^{(k)} - \bJ^\ast \| \le \|\bar\bJ^{(k)} - \bJ^{(k)}\|+\|\bJ^{(k)} - \bJ^\ast \| < C  \varepsilon^{(k)} + L \| \be_{k}\|.$$
Now, if $\|\be_{k}\| \le  \frac{1}{64\alpha \beta L}$, taking a tolerance $\varepsilon^{(k)}\le \frac{1}{64\alpha \beta C}$, we obtain
\begin{equation}\label{eq: diffJapJmin}
\|\bar\bJ^{(k)} - \bJ^\ast \|<\frac{1}{32\alpha \beta}.
\end{equation}
and so, if $\|\be_{k}\| \le  \frac{1}{4\beta L_{\mathcal{R}}}$, we deduce that
$\|\bar\bH^{(k)}-\bH^\ast\|<\frac{1}{2\beta}$
and, as a consequence,
\begin{align*}\|(\bH^\ast)^{-1} (\bar\bH^{(k)}-\bH^\ast)\|\le \|(\bH^\ast)^{-1} \| \| \bar\bH^{(k)}-\bH^\ast\|  < \frac{1}{2}.\end{align*}
Therefore, by \cite[Theorem 3.1.4]{dennisschnabel}, 
\begin{align*}
\|(\bar\bH^{(k)})^{-1}\|& \le  \frac{\|(\bH^\ast)^{-1}\|}{1-\|(\bH^\ast)^{-1} (\bar\bH^{(k)} - \bH^\ast)\|} 
< 2 \beta.
\end{align*}

From this inequality together with \eqref{eq:taylor}, \eqref{eq:gradient},  \eqref{eq:hessian}, and \eqref{eq: diffJapJmin},  it follows that
\begin{align*}
\| \be_{k+1}\|  & \le \|(\bar\bH^{(k)})^{-1}\| \Big(T \|\be_{k}\|^2 +  (C  + 4\kappa \alpha )\| \bb\| \varepsilon^{(k)}  + (2\alpha \|\bar\bJ^{(k)} - \bJ^\ast \| +\delta^\ast) \|\be_{k}\|\Big)\\
& \le (2 \beta T \| \be_k\| + 4\alpha \beta \|\bar\bJ^{(k)} - \bJ^\ast \| + 2\beta \delta^*)\| \be_k\|+ 2\beta (C + 4\kappa \alpha )\| \bb\| \varepsilon^{(k)} \\
& < \Big(2\beta T \|\be_{k}\|+ \frac{1}{8}+2\beta\delta^\ast\Big)\|\be_{k} \|+ 2\beta (C + 4\kappa \alpha )\| \bb\| \varepsilon^{(k)} .
\end{align*}

Set $\zeta:=\min\{ \frac{1}{64\alpha\beta L}, \frac{1}{4\beta L_{\mathcal{R}}}\frac{1}{32 \beta T}\} $. Assume, by induction, that
$$\|\be_{k}\| \le \min \left\{\zeta, \frac{1}{2^{k}}\right\}.$$
Then, the previously required conditions on $\|\be_k\|$ hold and so, as $\beta\delta^* <\frac{1}{32}$, 
we conclude that  
$$\| \be_{k+1}\| < \frac{1}{4} \|\be_{k} \| + 2\beta (C  + 4\kappa \alpha )\| \bb\| \varepsilon^{(k)}. $$
Now, if $\varepsilon^{(k)}\le \frac{1}{4\beta (C + 4\kappa \alpha )\| \bb\|} \min \left\{ \zeta, \frac{1}{2^{k+1}}\right\}$, it follows that 
\begin{align*}
\| \be_{k+1}\| &\le \frac{1}{4}  \min \left\{\zeta, \frac{1}{2^k}\right\} +  \frac{1}{2}   \min \left\{ \zeta, \frac{1}{2^{k+1}}\right\}  \le \min \left\{ \zeta, \frac{1}{2^{k+1}}\right\}.
\end{align*}
This finishes the proof. \end{proof}

\begin{remark}
Algorithms~\ref{Alg:RGenVarPro} and~\ref{Alg:I-GenVarPro} are analyzed under local convergence assumptions and are therefore presented using full quasi-Newton steps, i.e., $\by^{(k+1)} = \by^{(k)} + \bs^{(k)}$. In the numerical experiments of Section~\ref{section:NumericalExamples}, we implemented these methods without an explicit globalization strategy such as a line search or trust-region mechanism. This proved sufficient for the semi-blind deblurring problems considered, as the reduced problem is low-dimensional and exhibits mild nonlinearity, and no divergence or instability was observed even from relatively distant initial guesses. While globalization techniques could be incorporated to enhance robustness in more challenging settings, the present implementation already demonstrates stable convergence for the examples studied here.
\end{remark}

\section{Special Regularization Cases}\label{section:special_cases}
Now, we demonstrate the applicability of \linebreak \texttt{RGenVarPro} to two specific choices of the regularization term $\mathcal{R}(\by)$, which will be used in the numerical results section. 

\subsection{Regularization via 2-norm}\label{subsec:2-norm}

We first consider a quadratic regularization on $\by$ of the form
\begin{equation*}
\mathcal{R}(\by) = \frac{\mu^2}{2} \| \by - \by_0 \|^2,
\end{equation*}
whose gradient and Hessian are given by
\begin{equation*}
\nabla \mathcal{R}(\by) = \mu^2 (\by - \by_0),
\qquad
\nabla^2 \mathcal{R}(\by) = \mu^2 \mathbf{I}_r.
\end{equation*}

With this choice of regularization, the corresponding RSNLS problem takes the form
\begin{equation*}
\min_{\bx, \by} 
\frac{1}{2}\| \bA(\by) \bx - \bb \|^2
+ \frac{\lambda^2}{2} \| \bL \bx \|^2
+ \frac{\mu^2}{2} \| \by - \by_0 \|^2,
\end{equation*}
which yields the reduced optimization problem
\begin{align*}
\min_{\by}
\frac{1}{2}\| \bK(\by)\bK(\by)^\dagger \bd - \bd \|^2
+ \frac{\mu^2}{2} \| \by - \by_0 \|^2
&=
\min_{\by}\frac{1}{2} \left\| \begin{bmatrix} - \mathcal{P}_{\bK(\by)}^\perp \bd \\ \mu (\by - \by_0) \end{bmatrix} \right\|^2,
\end{align*}
where $\bK(\by)$ is defined as in~\eqref{eq: defKd}.

Define the nonlinear function
\begin{equation*}
\bF(\by) = 
\begin{bmatrix}
- \mathcal{P}_{\bK(\by)}^\perp \bd \\
\mu (\by - \by_0)
\end{bmatrix}.
\end{equation*}
We are now interested in solving the reduced nonlinear least squares problem
\begin{equation*}
\min_{\by} \frac{1}{2}\| \bF(\by) \|^2.
\end{equation*}
To solve this problem, we apply the Gauss–Newton method, with iterates defined by
\begin{equation*}
\by^{(k+1)} = \by^{(k)} + \bs^{(k)}, \quad k = 0, 1, 2, \dots,
\end{equation*}
where $\bs^{(k)}$ is the Gauss–Newton step obtained by solving
\begin{equation*}
\bs^{(k)} = \arg\min_{\bs} \left\| \bJ_{\bF}(\by^{(k)}) \bs + \bF(\by^{(k)}) \right\|^2,
\end{equation*}
and $\bJ_{\bF}(\by) \in \mathbb{R}^{(m+q+r)\times r}$ denotes the Jacobian of $\bF$.

Notice that the Jacobian matrix $\bJ_{\bF}(\by)$ has the block structure
\begin{equation*}
\bJ_{\bF}(\by) = 
\begin{bmatrix}
\bJ_{\bff}(\by) \\
\mu \mathbf{I}_r
\end{bmatrix},
\end{equation*}
where $\bff(\by)= - \mathcal{P}_{\bK(\by)}^\perp \bd$.

\subsection{Regularization via Logarithms}\label{subsec:log}
The other regularization for $\by$ that we consider is inspired by the log-barrier optimization method~\cite{nocedal2006numerical}, and therefore, we propose solving the RSNLS problem
\begin{equation*}
    \min_{\bx, \by} \frac{1}{2}\| \bA(\by) \bx - \bb \|^2 + \frac{\lambda^2}{2} \| \bL \bx\|^2 - \sum_{1\le j \le r}\mu_j^2 \log(y_j)\end{equation*} 
with positive constants $\mu_j$, $j=1,\dots, r$.

The corresponding reduced functional has the form
\[ \varphi(\by) = \dfrac{1}{2}\| - \calP_{\bK}^\perp(\by) \bd \|^2 - \sum_{1\le j \le r}\mu_j^2 \log(y_j). \]
To solve $\min_{\by} \varphi(\by)$, we may apply a quasi-Newton's method using
\[ \nabla \varphi(\by) = \bJ^\top_{\mathbf{f}}(\by) \mathbf{f}(\by)-\left[\begin{array}{cccc} \mu_1^2/y_1 & \dots & \mu_r^2/y_r   \end{array} \right]^\top.\]
and the approximated Hessian given by
\[ \bH(\by) = \bJ^\top_{\mathbf{f}}(\by) \bJ_{\mathbf{f}}(\by) + \bD^\top(\by) \bD(\by), \]
where $\bD(\by)$ is the diagonal matrix with  $[\bD(\by)]_{jj} = -\mu_j / y_j$, $j=1,\dots, r$.
Thus, the iterations are defined by 
$$
\by^{(k+1)} = \by^{(k)} + \bs^{(k)}, \,k = 0,1,2,...,
$$
where
\begin{equation*}
\bs^{(k)} = \arg\min_{\bs} \left\| \left[\begin{array}{c}\bJ_{\bff}(\by^{(k)}) \\ \bD(\by^{(k)})\end{array}\right]\bs + \left[\begin{array}{c} \bff(\by^{(k)}) \\ \mathbf{u}\end{array} \right] \right\|^2
\end{equation*}
with $\mathbf{u}= [\mu_1, \dots, \mu_r]^\top$.

\section{Numerical Results}\label{section:NumericalExamples}

In this section, we validate the proposed algorithms on a semi-blind image deblurring problem. The goal is to recover both the sharp image and the blur parameters governing the point spread function (PSF). We test the accuracy, convergence, and computational performance of \texttt{RGenVarPro} and \texttt{iRGenVarPro}.  Our implementations of the methods described in this paper will be available at \url{https://github.com/malenaespanol/RGenVarPro}.

\subsection{Experimental Setup}

The blurring matrix $\bA(\by)$ is ill-conditioned and depends on a parameterized PSF $\bP(\by)\in\mathbb{R}^{n\times n}$. 
We consider an \emph{isotropic} Gaussian PSF depending on a single scalar parameter $\by>0$, defined by
\begin{equation*}
[\bP(\by)]_{i,j} = c(\by)\exp\!\left(-\frac{(i-k)^2 + (j-\ell)^2}{2\by^2}\right),
\end{equation*}
where the normalization constant $c(\by)$ (which depends on $\by$) satisfies $\sum_{i,j} [\bP(\by)]_{i,j} = 1$. 
Then, the entries of $\bA(\by)$ are defined by $\bP(\by)$. This isotropic formulation is sufficient to illustrate the main phenomena of interest, namely, the occurrence of degenerate no-blur solutions in the absence of regularization on the nonlinear parameters and the stabilizing effect of the proposed regularization strategies, while keeping the reduced optimization problem one-dimensional and allowing for clear visualization of the objective function $\varphi(\by)$.

For the regularization operator $\bL$, we employ the standard five-point Laplacian stencil
\begin{equation*}
\begin{bmatrix} 
0 & 1 & 0 \\ 
1 & -4 & 1\\ 
0 & 1 & 0
\end{bmatrix},
\end{equation*}
to construct the discrete Laplacian matrix $\bL$. Assuming periodic boundary conditions and spatially invariant blur, both $\bA$ and $\bL$ are jointly diagonalizable via the discrete Fourier transform. Efficient implementations for computing residuals and Jacobians follow \cite{Espanol_2023}.

\begin{figure}[ht]
	\centering
	\begin{minipage}{0.32\textwidth}
		\centering
	  \includegraphics[width = \textwidth]{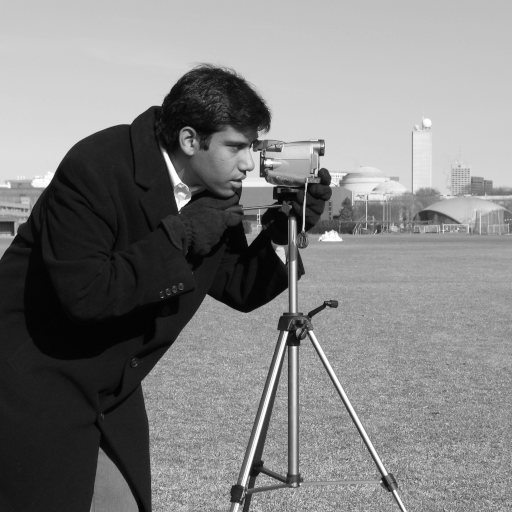}\\(a)
	\end{minipage}
	\begin{minipage}{0.32\textwidth}
		\centering
	  \includegraphics[width = \textwidth, height = \textwidth]{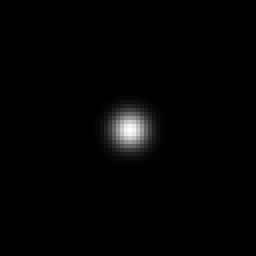}\\(b)
	\end{minipage}
	\begin{minipage}{0.32\textwidth}
		\centering
	  \includegraphics[width = \textwidth, height = \textwidth]{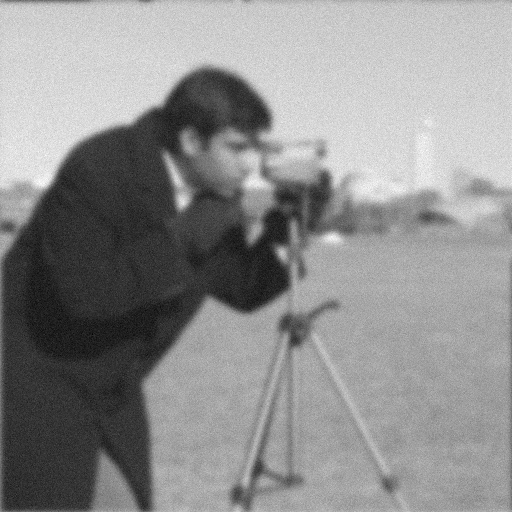}\\(c)
	\end{minipage}
	\caption{Illustrative \textit{cameraman} test problem: (a) true image, (b) PSF (cropped and upsampled by a factor of $4$ via nearest-neighbor interpolation), and (c) blurred and noisy image with $5\%$ Gaussian white noise.}
	\label{fig:Setup}
\end{figure}

We use the $512\times 512$ \textit{cameraman} image (Figure~\ref{fig:Setup}(a)) as the ground truth. The image is blurred using a Gaussian PSF with parameter $\by_{\mathrm{true}}=3$. Gaussian white noise with a standard deviation corresponding to $5\%$ of the image norm is then added, producing the observation in Figure~\ref{fig:Setup}(c).

\subsection{Testing \texttt{RGenVarPro}}
We solve the deblurring problem using \texttt{RGenVarPro} (Algorithm~\ref{Alg:RGenVarPro}) with initial guess $\by^{(0)}=5$. All computations were implemented in Python and executed on Google Colab. To assess the reconstruction quality at iteration $k$, we compute the Relative Reconstruction Error (RRE) with respect to both $\bx$ and $\by$, which are defined as follows
\begin{equation*}
{\rm RRE}(\bx^{(k)}) = \frac{\|\bx^{(k)} - \bx_{\rm true}\|}{\|\bx_{\rm true}\|} \mbox{ and } {\rm RRE}(\by^{(k)}) = \frac{|\by^{(k)} - \by_{\rm true}|}{|\by_{\rm true}|}.
\end{equation*}

We use parameter values $\lambda = 1.5$ and $\mu = 3.8$ when $\mathcal{R}(\by)$ is the 2-norm regularizer (see Subsection~\ref{subsec:2-norm}), and $\lambda = 0.425$ and $\mu = 3.8$ when $\mathcal{R}(\by)$ follows the logarithmic regularizer (see Subsection~\ref{subsec:log}). The regularization parameters $\lambda$ and $\mu$ were chosen empirically based on preliminary tests to illustrate the effect of regularization on both the linear and nonlinear variables. The parameter $\lambda$ controls the amount of smoothing in the reconstructed image, while $\mu$ governs the strength of the penalty on the blur parameter and prevents convergence to the trivial no-blur solution. Although no systematic parameter selection strategy (such as the L-curve or cross-validation) was employed, we observed that the proposed methods exhibit stable convergence and similar qualitative behavior over a range of moderate values of $\lambda$ and $\mu$. A detailed study of automated parameter selection is beyond the scope of this work. Figure~\ref{fig:model} compares the objective function $\varphi(\by)$ for the unregularized and regularized cases. As seen, the unregularized model attains its minimum at $\by=0$, while the regularized models achieve minima near $\by=3$, which matches the true blur parameter.

\begin{figure}[ht]
	\centering
	\begin{minipage}{0.32\textwidth}
	  \includegraphics[width = \textwidth]{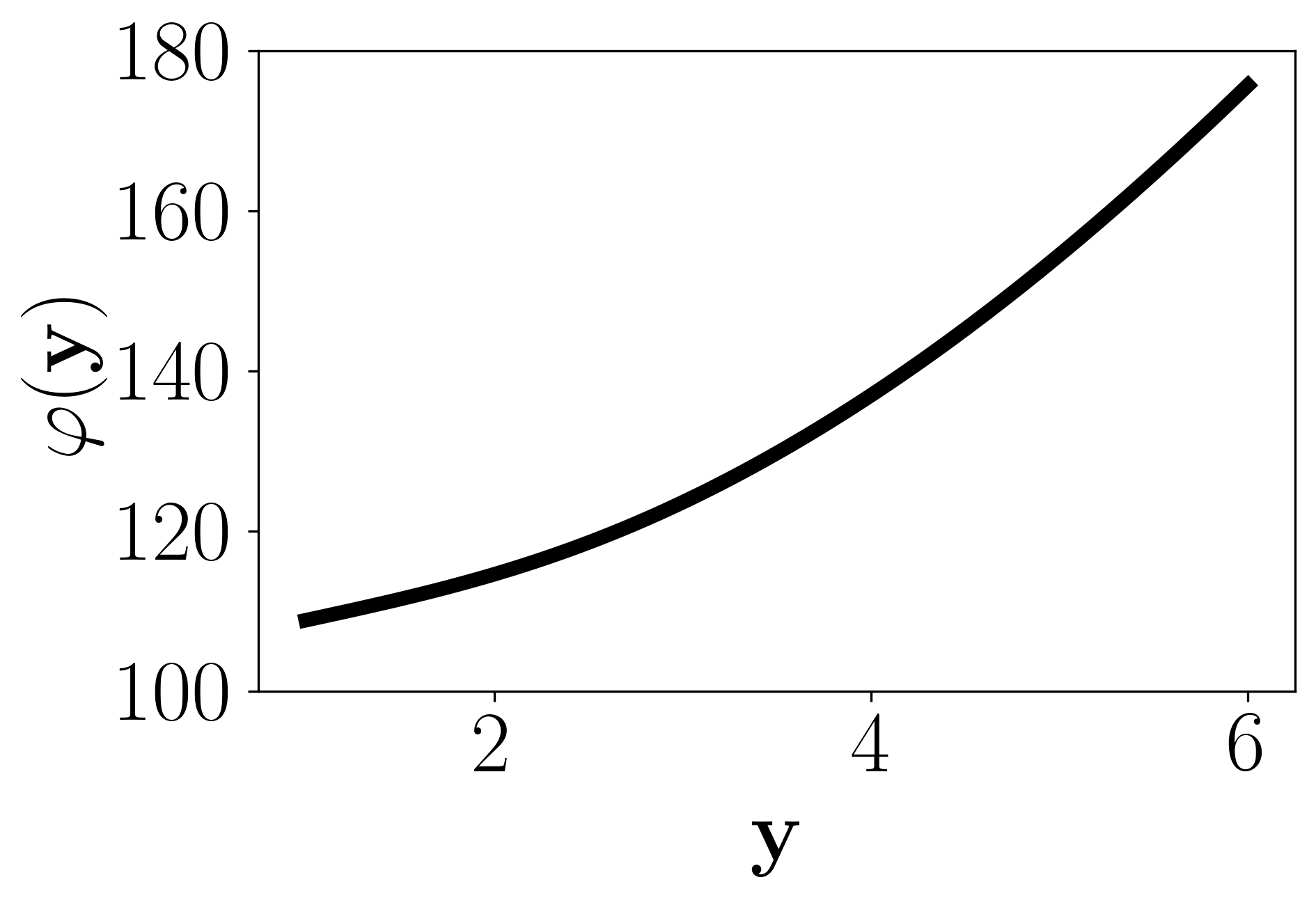}
	\end{minipage}
	\begin{minipage}{0.32\textwidth}
	  \includegraphics[width = \textwidth]{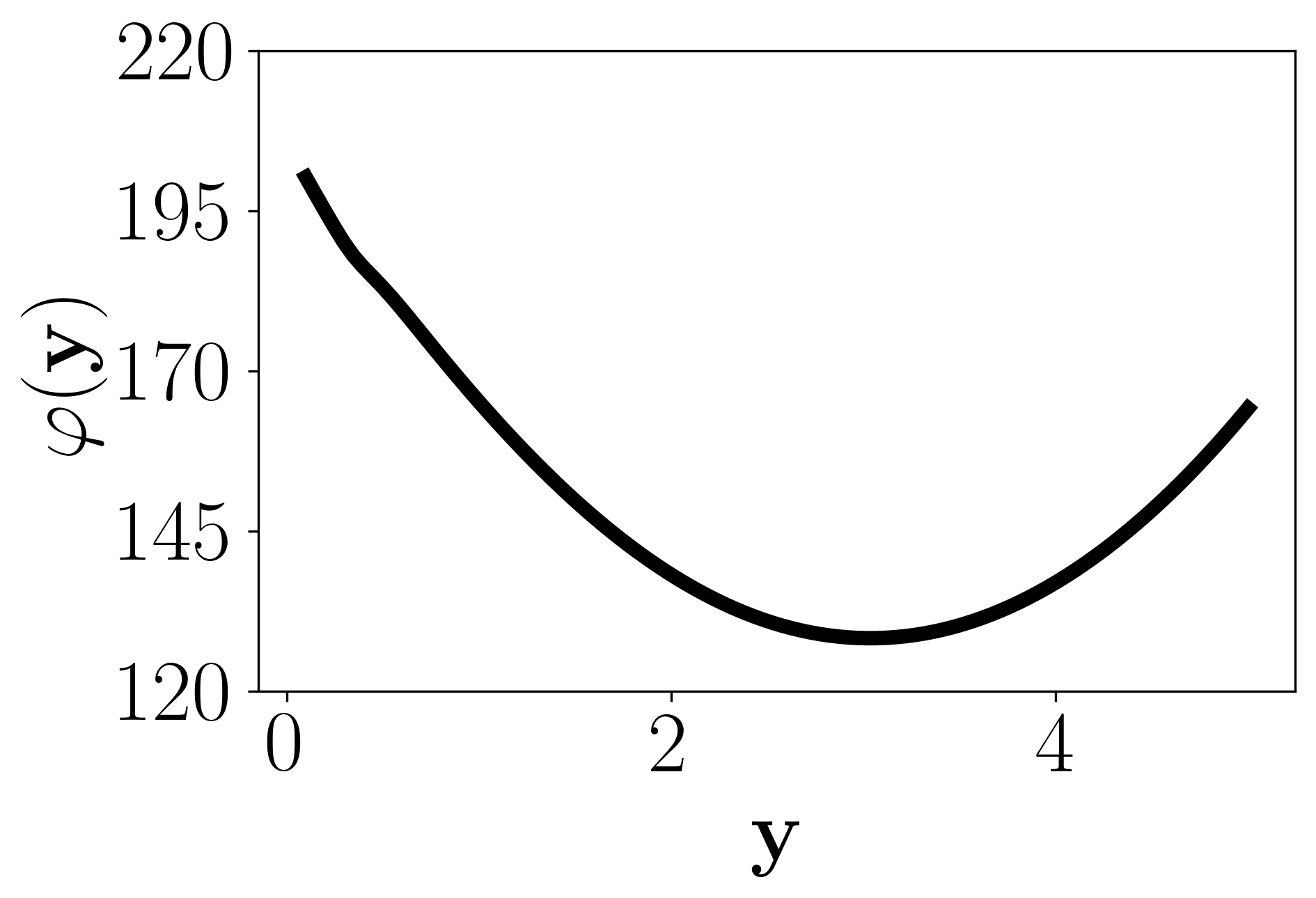}
	\end{minipage}
	\begin{minipage}{0.32\textwidth}
	  \includegraphics[width = \textwidth]{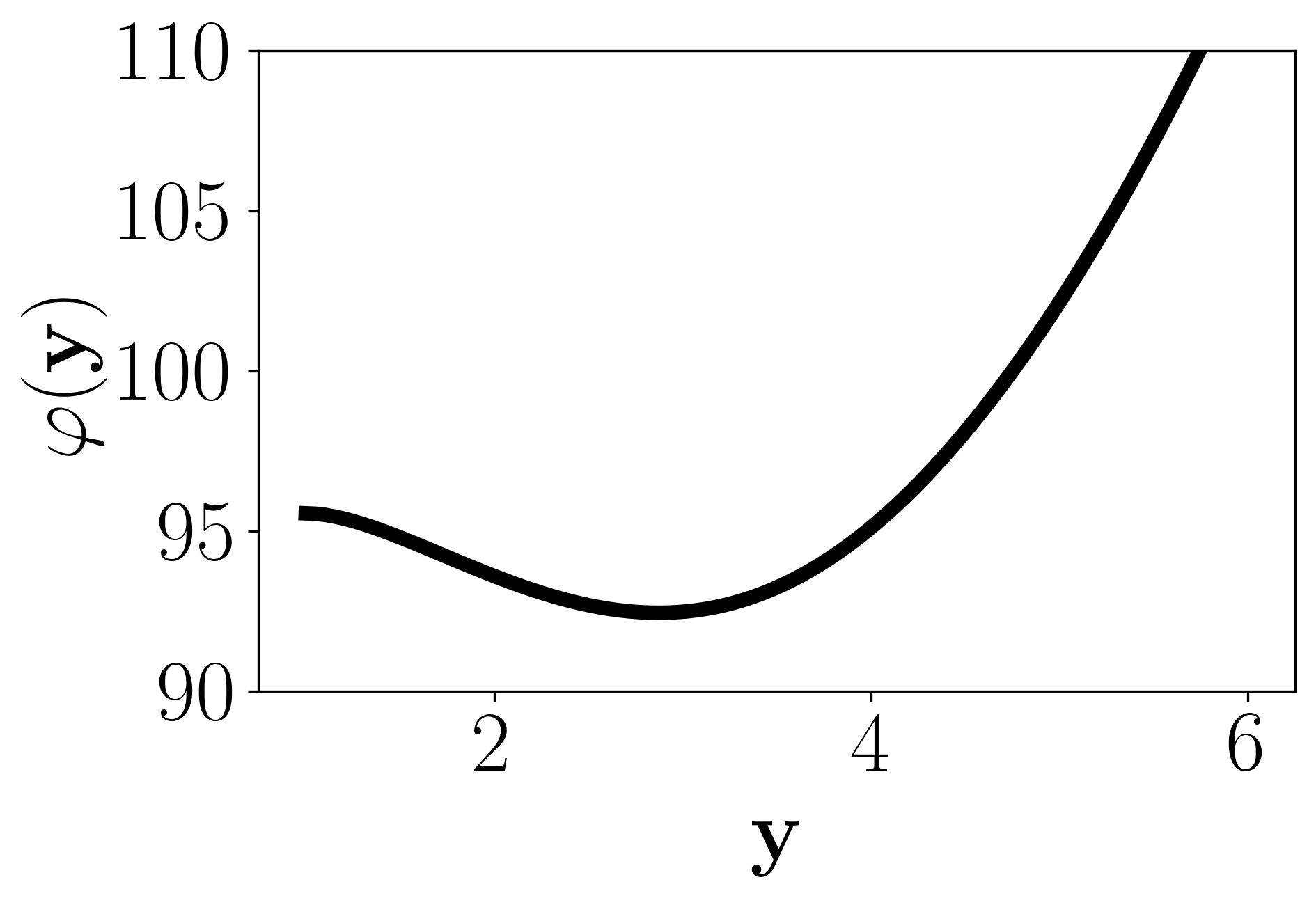}
	\end{minipage}
	\caption{Plots of $\varphi(\by)$ without regularization (left), with $\mathcal{R}(\by)= \mu^2\|\by - \by_0\|^2$ (middle), and with $\mathcal{R}(\by)=-\sum_{j}\mu_j^2\log(y_j)$ (right). Regularization shifts the minimum from $\by=0$ toward the true value $\by\approx3$.}
	\label{fig:model}
\end{figure}

Figures~\ref{fig:convergence_norm} and~\ref{fig:convergence_log} illustrate the convergence behavior in terms of the relative reconstruction error and the evolution of the objective function $\mathcal{F}(\bx^{(k)}, \by^{(k)})$ for the 2-norm and logarithmic regularizers, respectively. Both regularizations recover a visually accurate image as can be seen in Figure~\ref{fig:reconstructions} and achieve stable convergence.

\begin{figure}[ht]
	\centering
	\begin{minipage}{0.32\textwidth}
	  \includegraphics[width = \textwidth]{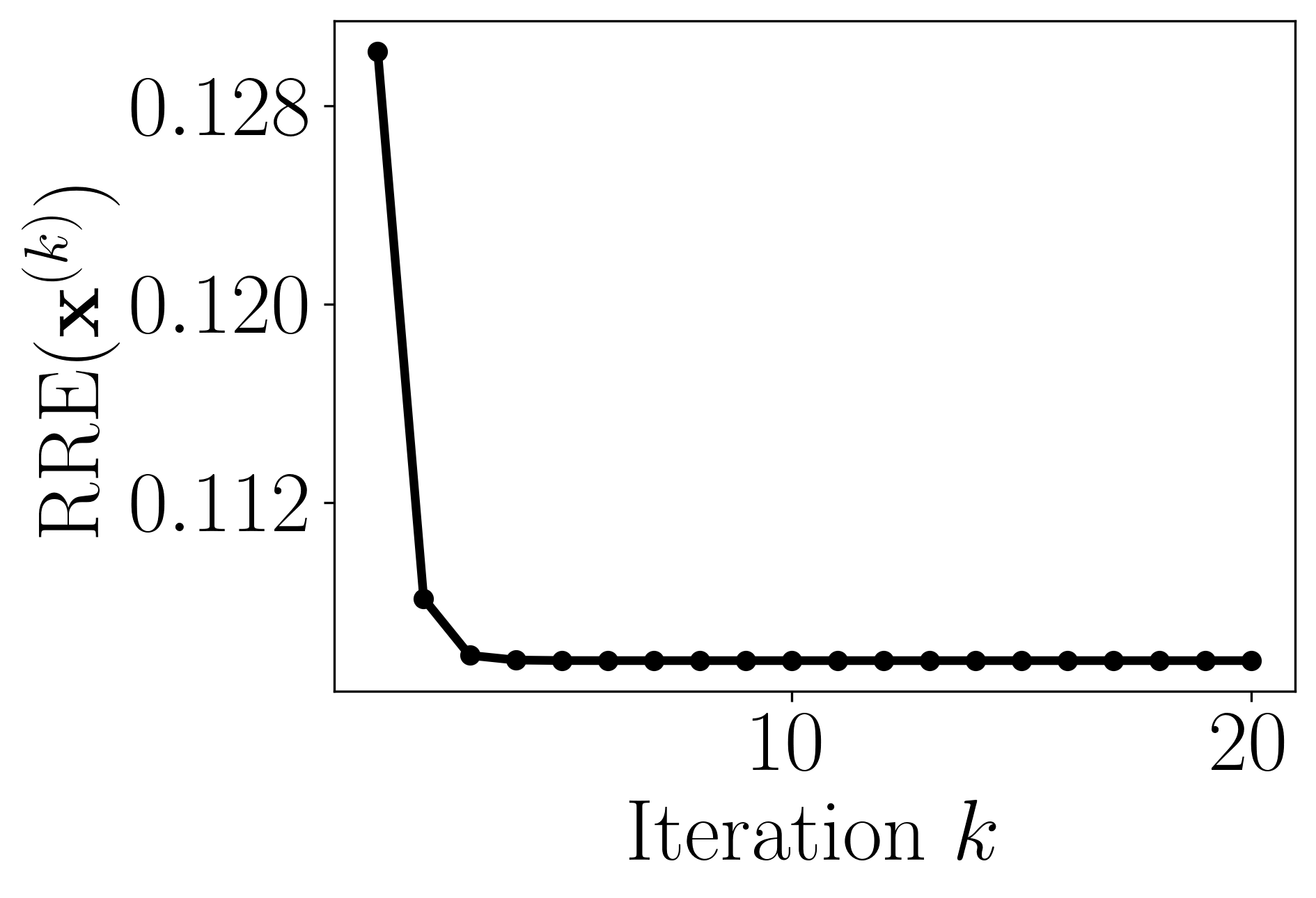}
	\end{minipage}
	\begin{minipage}{0.32\textwidth}
	  \includegraphics[width = \textwidth]{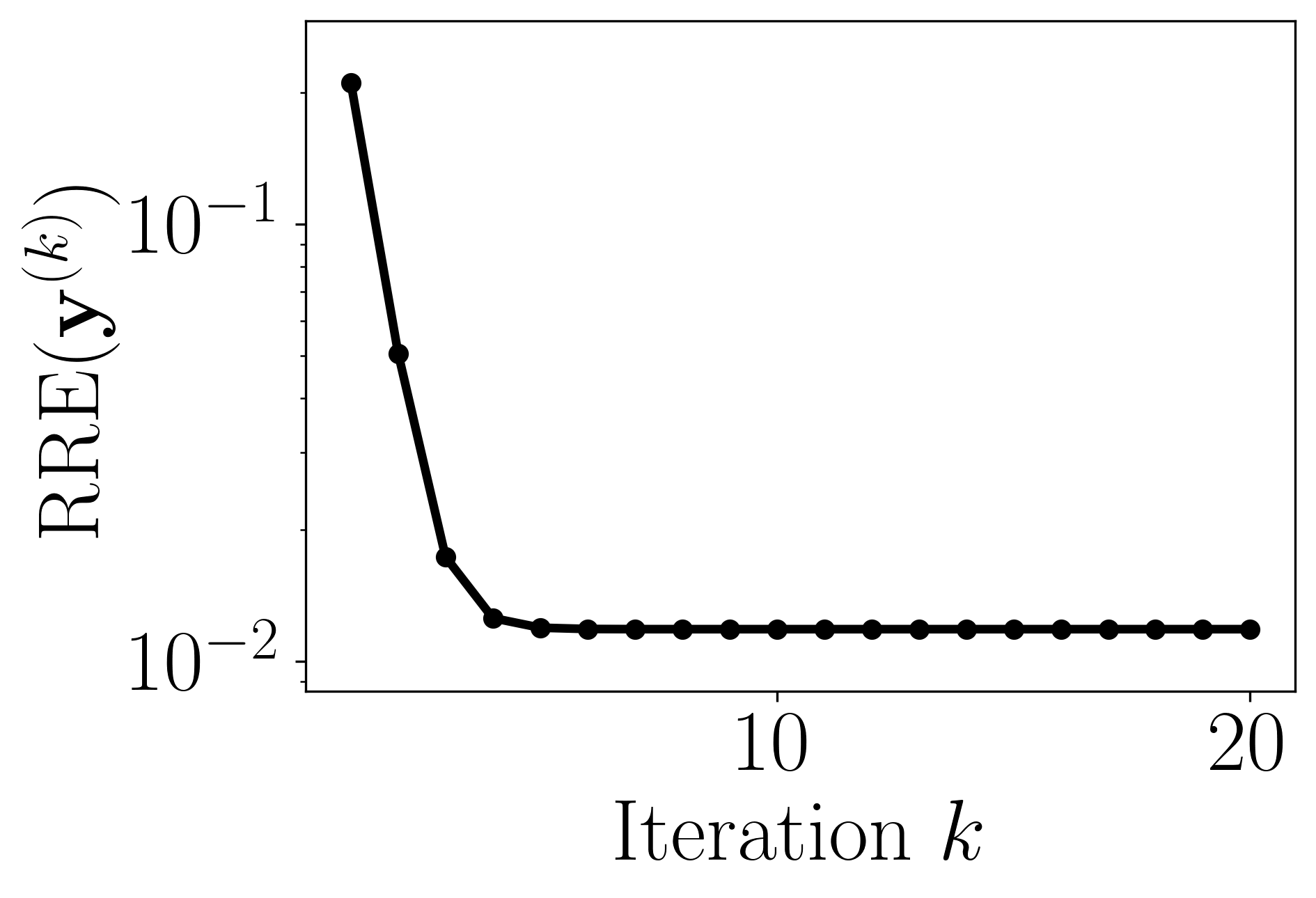}
	\end{minipage}
    \begin{minipage}{0.32\textwidth}
	  \includegraphics[width = \textwidth]{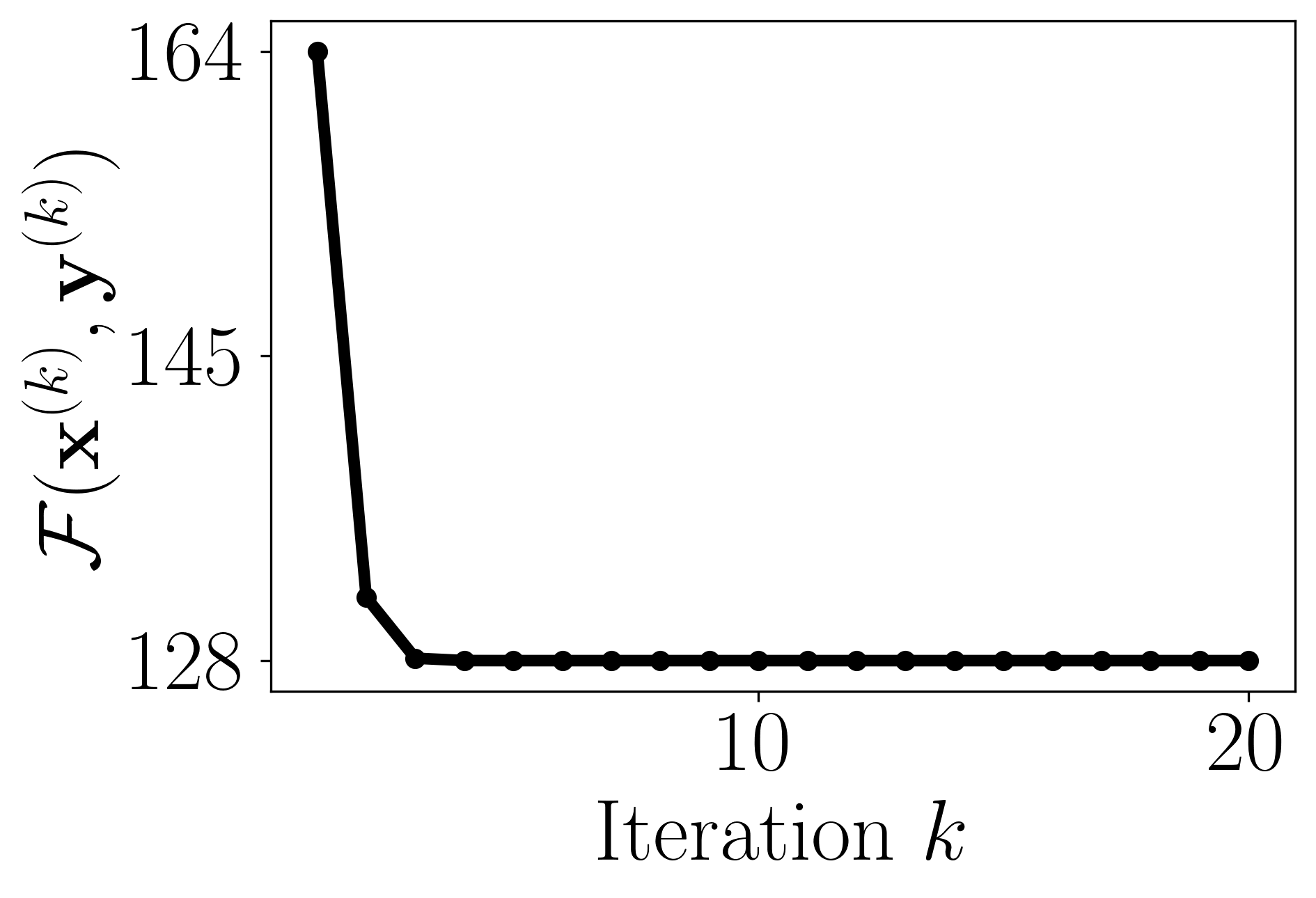}
	\end{minipage}
	\caption{Results for the 2-norm regularizer: (left) reconstructed image (SSIM = 0.66), (middle) RRE for $\bx$, and (right) objective function value vs.\ iteration number.}
	\label{fig:convergence_norm}
\end{figure}

\begin{figure}[ht]
	\centering
	\begin{minipage}{0.32\textwidth}
	  \includegraphics[width = \textwidth]{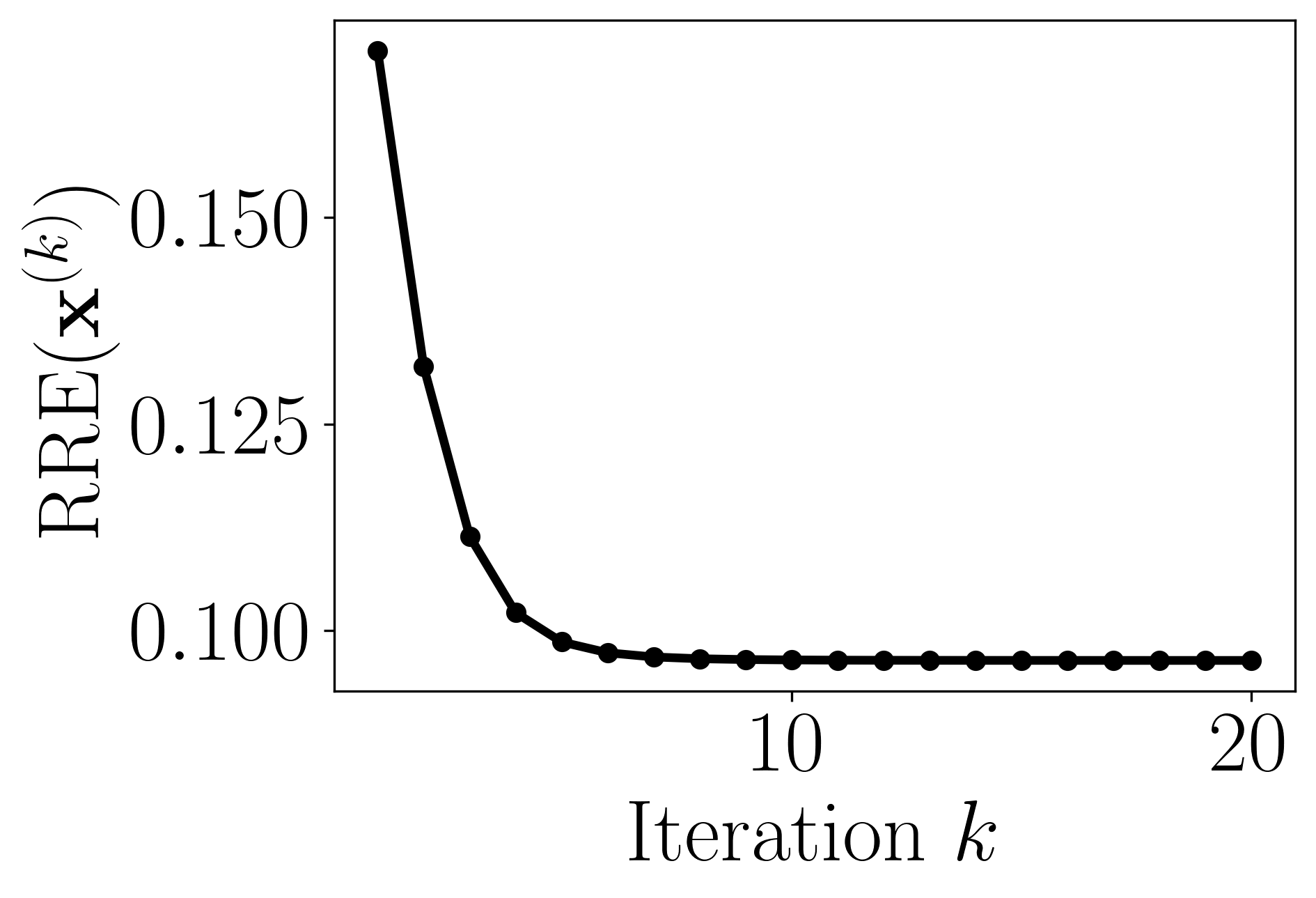}
	\end{minipage}
	\begin{minipage}{0.32\textwidth}
	  \includegraphics[width = \textwidth]{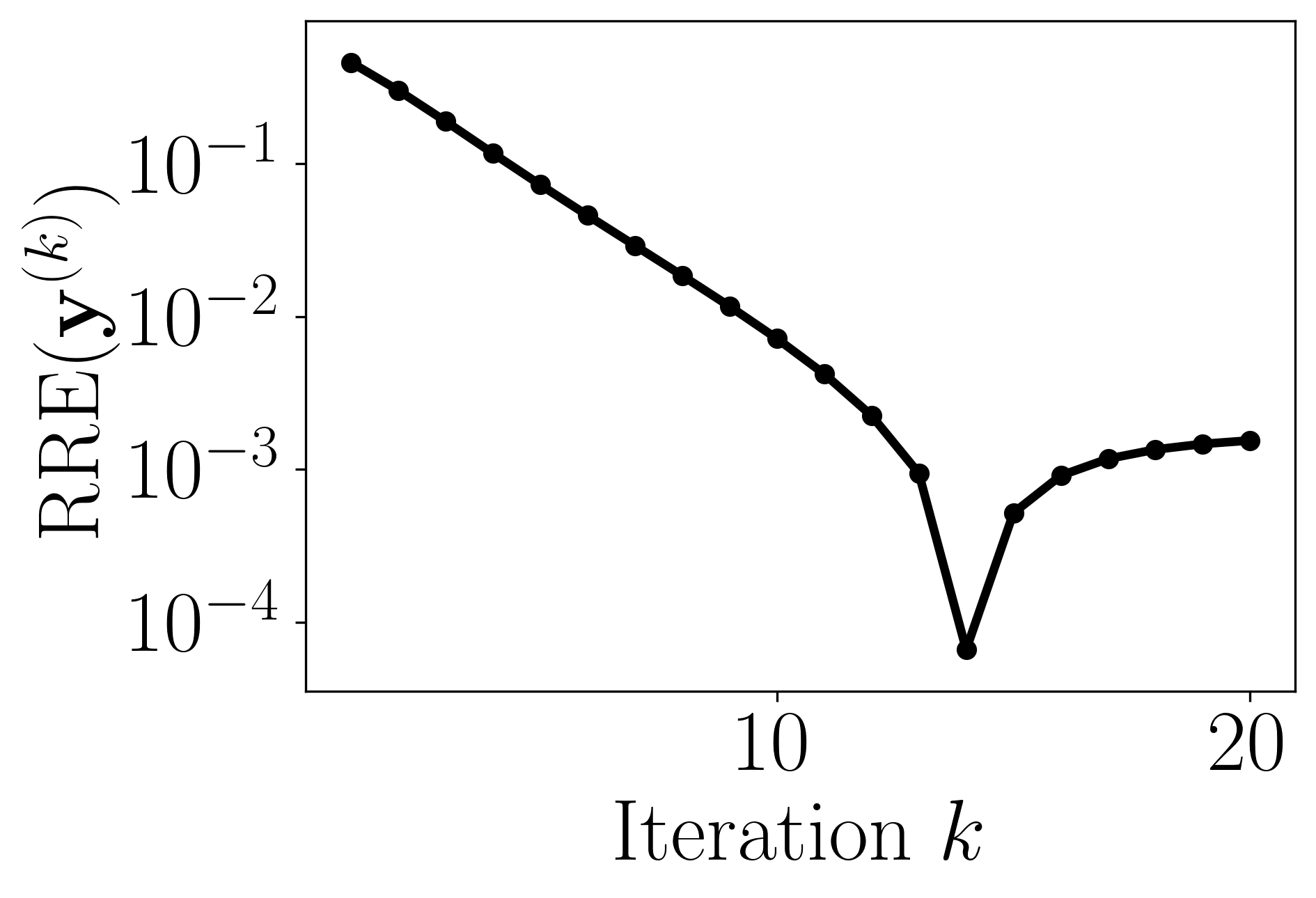}
	\end{minipage}
    \begin{minipage}{0.32\textwidth}
	  \includegraphics[width = \textwidth]{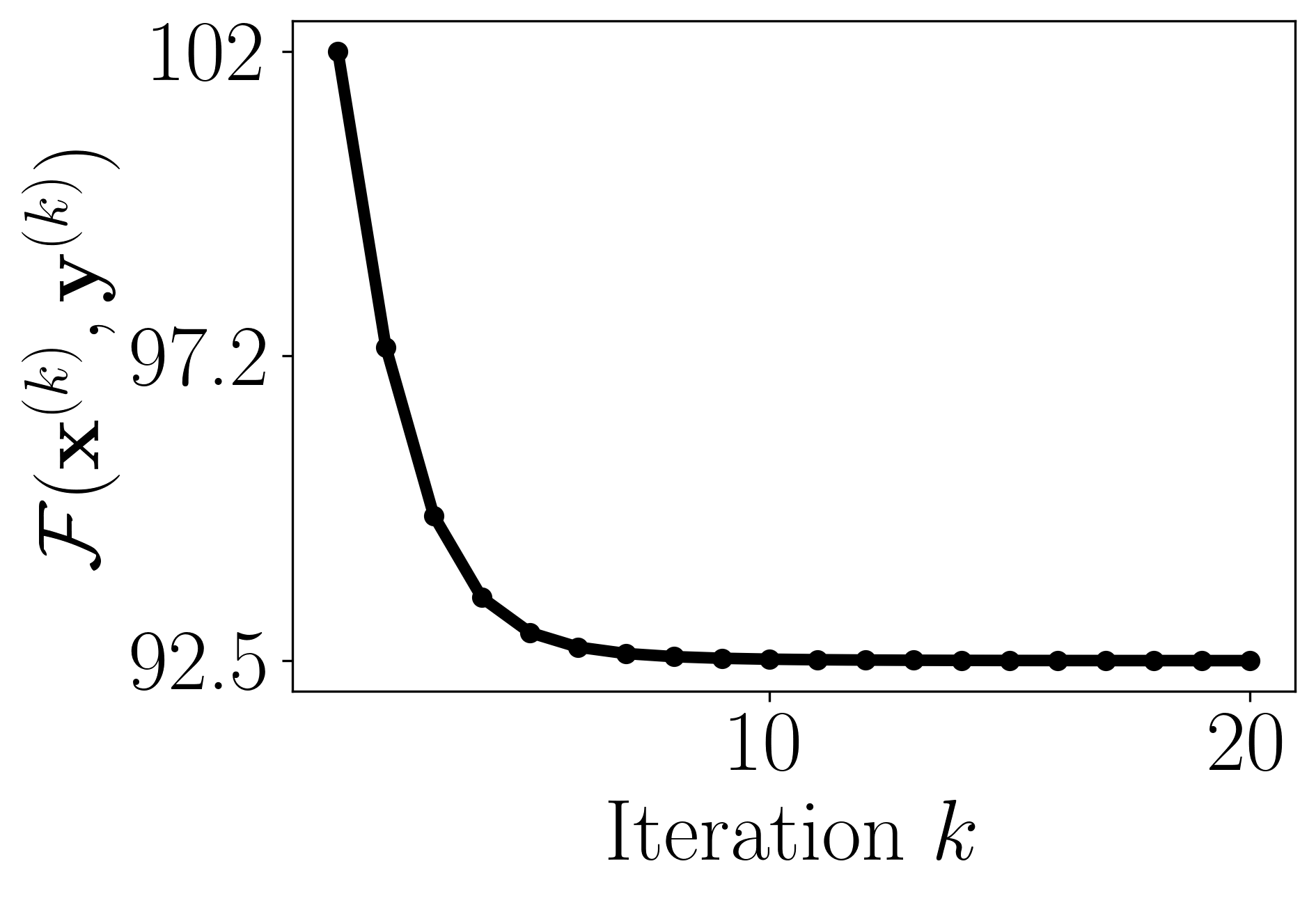}
	\end{minipage}
	\caption{Results for the logarithmic regularizer: (left) reconstructed image (SSIM = 0.63), (middle) RRE for $\bx$, and (right) objective function value vs.\ iteration number.}
	\label{fig:convergence_log}
\end{figure}

\subsection{Testing \texttt{iRGenVarPro}}

Next, we solve the semi-blind deblurring problem using the inexact variant \texttt{iRGenVarPro} (Algorithm~\ref{Alg:I-GenVarPro}) with initial guess $\by^{(0)}=5$. The inner linear system is solved \emph{approximately} using the Python implementation of the LSQR algorithm, available as \texttt{scipy.sparse.linalg.lsqr}, with a maximum of 300 iterations. We compare four tolerance strategies for the LSQR subproblem:
\begin{itemize}
\item \textbf{LSQR-s:} fixed small tolerance $\varepsilon^{(k)}=10^{-9}$;
\item \textbf{LSQR-ab:} exponentially decreasing tolerance $\varepsilon^{(k)}=\varepsilon^{(k-1)}/2$;
\item \textbf{LSQR-lb:} linearly decreasing tolerance $\varepsilon^{(k)}=\varepsilon^{(0)}/k$;
\item \textbf{LSQR-b:} fixed large tolerance $\varepsilon^{(k)}=\varepsilon^{(0)}$.
\end{itemize}
We use $\varepsilon^{(0)}=10^{-3}$ for LSQR-ab, LSQR-lb, and LSQR-b. Note that LSQR-ab corresponds to the tolerance strategy used in Algorithm~\ref{Alg:I-GenVarPro}, for which the theoretical convergence results in Theorem~\ref{thm: main} apply.

\begin{figure}[ht]
    \centering
    \includegraphics[width=.47\linewidth]{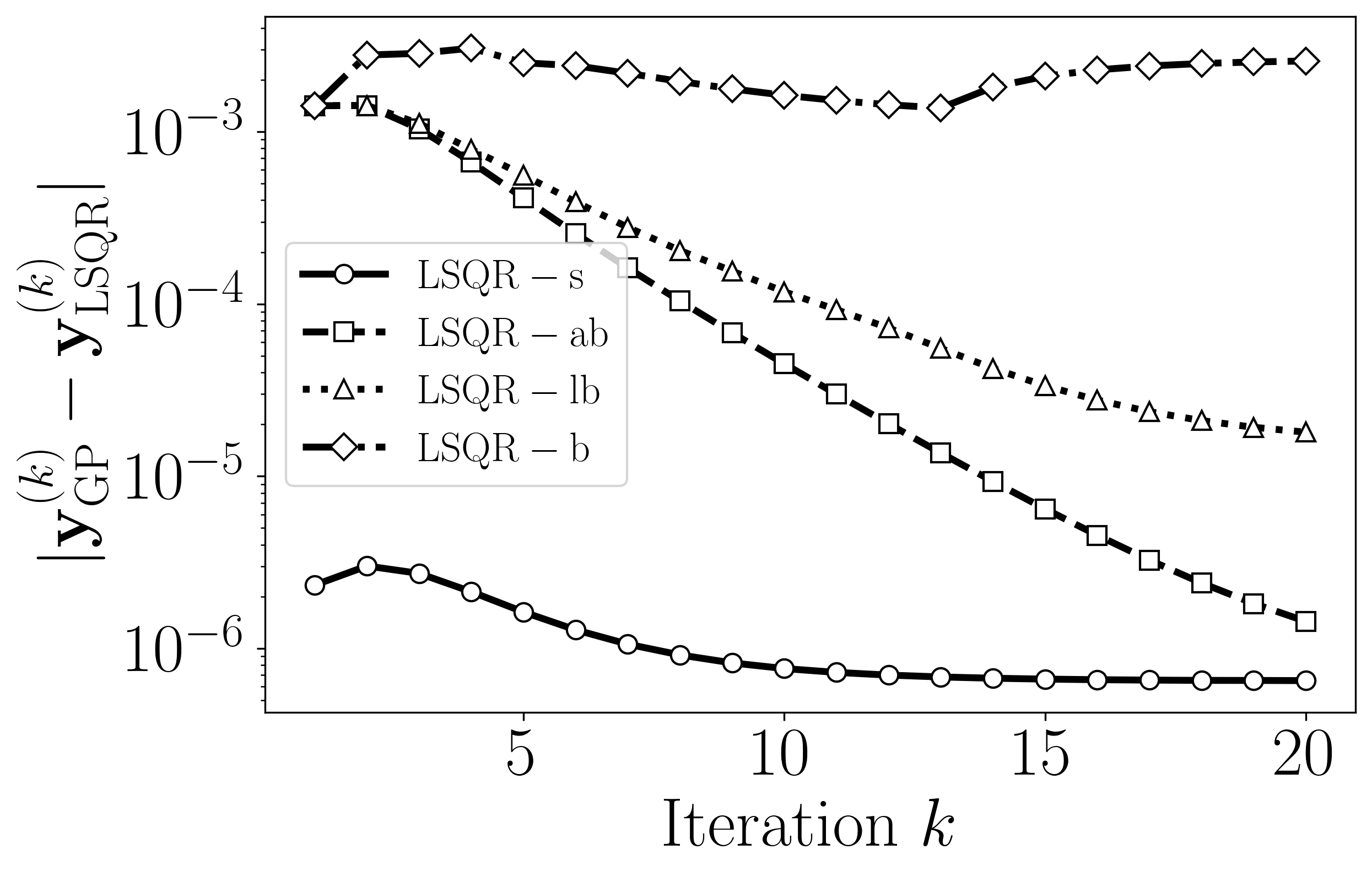}
    \includegraphics[width=.47\linewidth]{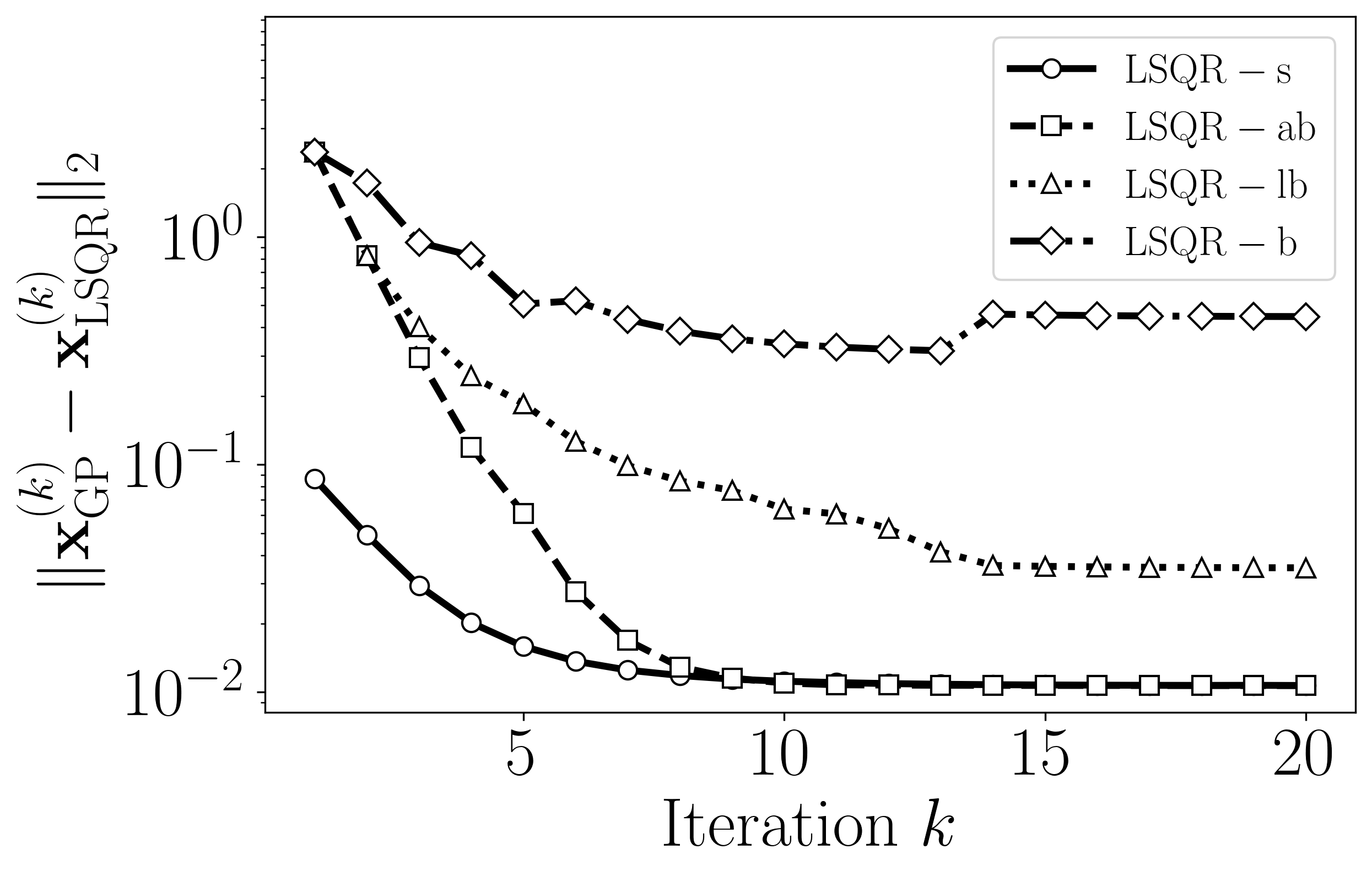}\\[0.5ex]
    \includegraphics[width=.47\linewidth]{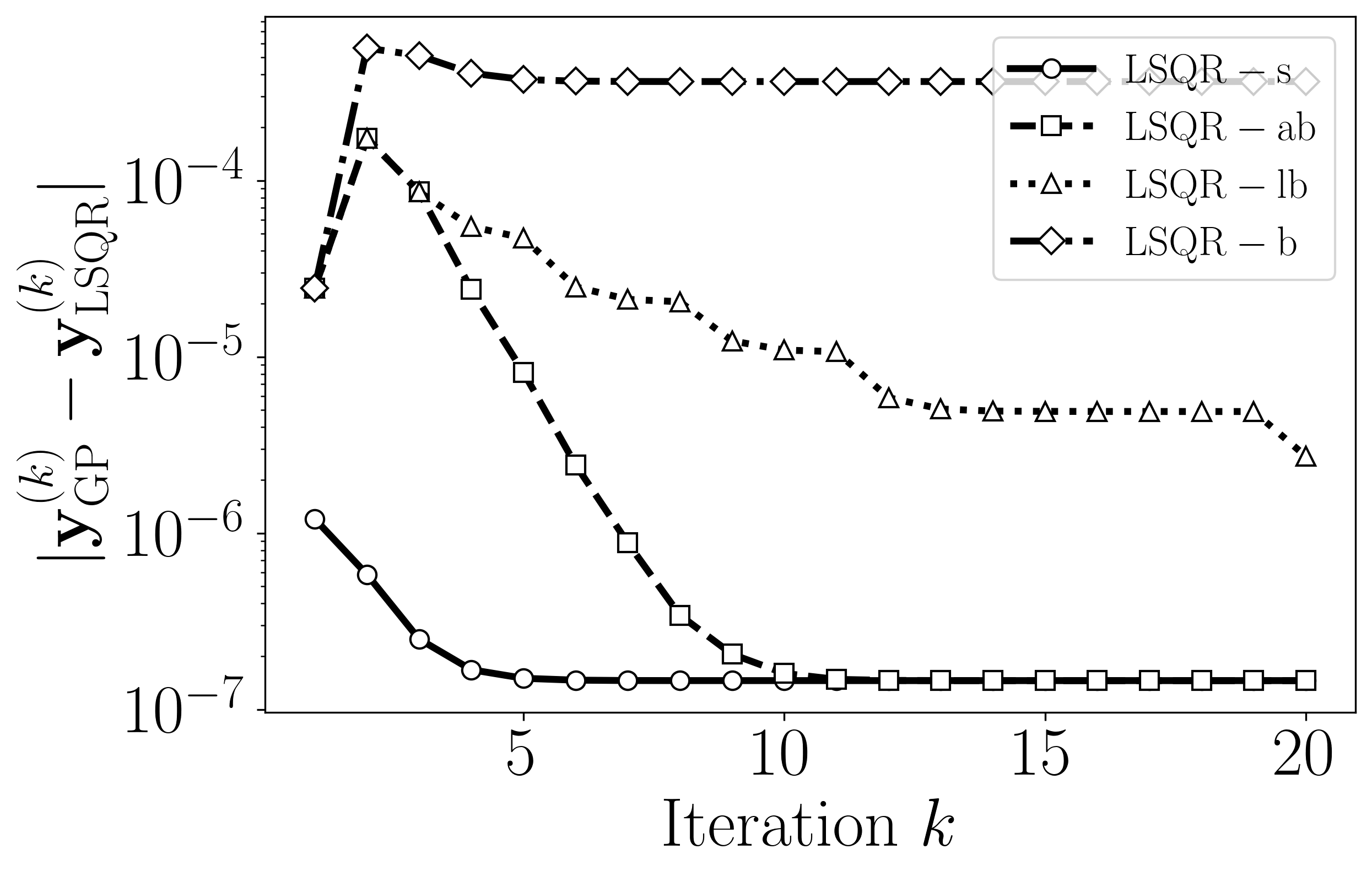}
    \includegraphics[width=.47\linewidth]{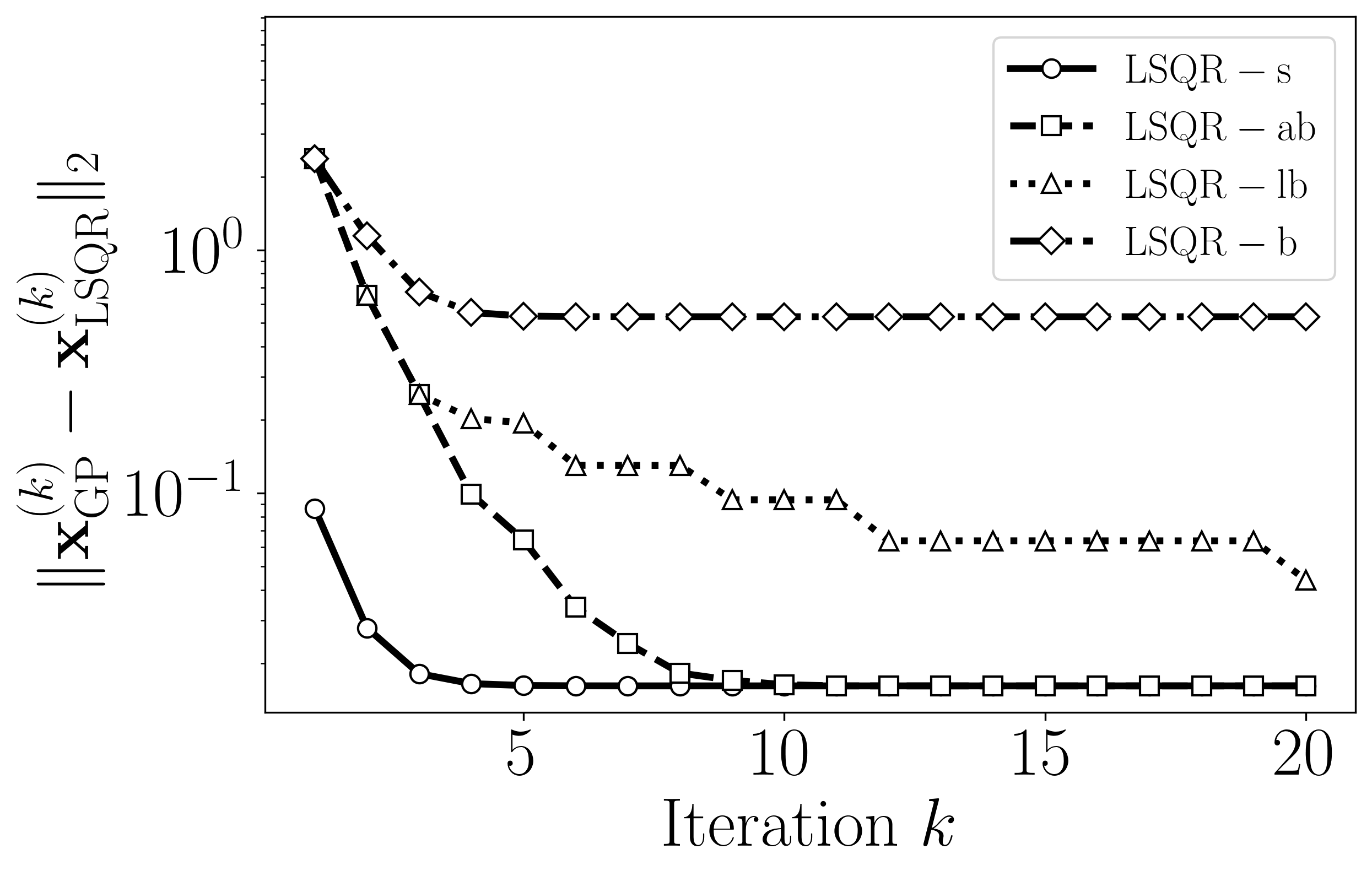}
    \caption{Comparison of \texttt{iRGenVarPro} results for different LSQR tolerance strategies. Errors with respect to the solution obtained using \texttt{RGenVarPro} in $\by$ and $\bx$ at iteration $k$ for the 2-norm regularizer (top) and the logarithmic regularizer (bottom).}
    \label{fig:tol_schedules}
\end{figure}

Figure~\ref{fig:tol_schedules} shows the evolution of the absolute errors in $\by$ and $\bx$ relative to the solutions obtained with \texttt{RGenVarPro}~(GP) across iterations for different tolerance strategies, while Table~\ref{tab:time} reports the corresponding computational times for 30 iterations. As expected, looser tolerances (LSQR-b and LSQR-lb) substantially reduce computational cost at the expense of reconstruction accuracy. Conversely, smaller tolerance values (LSQR-ab and LSQR-lb) yield more accurate results, with the adaptive exponential strategy (LSQR-ab) offering the best trade-off between speed and precision, as it can be terminated earlier without significant loss of accuracy.

\begin{figure}[ht]
    \centering
    \includegraphics[width=1\linewidth]{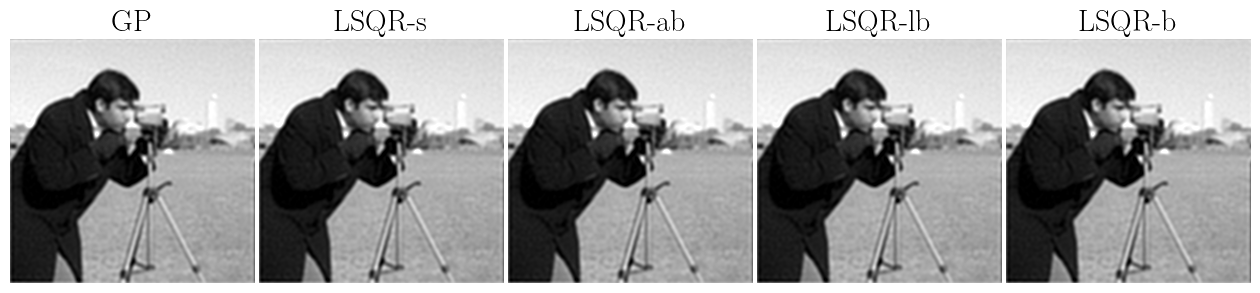}\\[0.5ex]
    \includegraphics[width=1\linewidth]{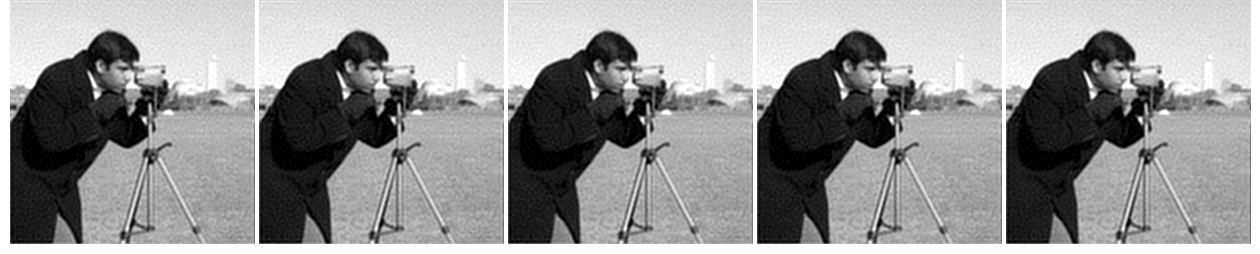}
    \caption{Reconstructed images obtained using \texttt{RGenVarPro} \emph{(GP)} and \texttt{iRGenVarPro}  under different LSQR tolerance strategies \emph{(LSQR-s}, \emph{LSQR-ab}, \emph{LSQR-lb}, and \emph{LSQR-b)} for the 2-norm (top) and logarithmic (bottom) regularizers.}
    \label{fig:reconstructions}
\end{figure}

\begin{table}[!ht]
    \centering
    \caption{CPU time (in seconds) required for seven iterations of \texttt{iRGenVarPro} under different LSQR tolerance strategies.}
    \label{tab:time}
    \begin{tabular}{l|cccc}
        & LSQR-s & LSQR-ab & LSQR-lb & LSQR-b \\ \hline
        2-norm model & 114 & 90 & 55 & 38 \\
        Log model    & 110 & 80 & 53 & 34 \\
    \end{tabular}
\end{table}

The numerical experiments demonstrate the effectiveness of the proposed variable projection methods for semi-blind image deblurring. The regularized formulations successfully recover both the image and blur parameters, while the unregularized model fails to identify the correct blur width. The \texttt{RGenVarPro} algorithm achieves stable convergence and high-quality reconstructions, whereas its inexact counterpart, \texttt{iRGenVarPro}, provides a flexible balance between computational cost and accuracy. Among the tested tolerance strategies, the exponentially decreasing scheme (\texttt{LSQR-ab}) achieves performance comparable to that of the exact solver at a significantly reduced cost. Overall, these experiments demonstrate that the proposed algorithms are both accurate and efficient, providing strong empirical support for the theoretical results. 

\section{Conclusions}\label{section: conclusions}
We presented an extension of the variable projection (VarPro) method for solving separable nonlinear least squares problems with regularization on both the nonlinear and linear parameters. More specifically, our approach solves problems with general-form Tikhonov regularization for the linear variables and a differentiable regularization term for the nonlinear parameters, thereby improving stability and reconstruction quality in inverse problems such as semi-blind image deblurring. 
To efficiently solve the reduced problem introduced through the VarPro method, we developed a quasi-Newton method. We established a local convergence analysis under standard smoothness assumptions, demonstrating conditions for superlinear or quadratic convergence. For large-scale problems, we proposed using iterative solvers, such as LSQR, to compute approximate solutions to the inner linear subproblems and to approximate the Hessian and residuals. We proved local convergence of the quasi-Newton method even under these approximations. 

Our numerical experiments on semi-blind image deblurring confirmed the theoretical results, showing that the proposed method effectively recovers both the image and blur parameters. These results also indicate the robustness of the method across different approximation strategies and its potential for large-scale practical applications.  

Future work includes extending this framework to dynamic inverse problems and exploring its integration within Bayesian inverse problem formulations and learning-based approaches, where parameter estimation and uncertainty quantification can be addressed simultaneously.

\appendix
\section{Generalized Motivating Example}\label{appendix}
Here, we extend the analysis done for the $2\times 2$ motivating example in Section \ref{sec: motivating example}. We now assume that the matrix \(\bA(\sigma)\in \mathbb{R}^{n\times n}\) is defined by some kernel \(g_\sigma\) that depends on \(\sigma\in \mathbb{R}\) and  is diagonalizable by the unitary Discrete Fourier Transform (DFT) matrix  \(\bF\) (so \(\bF^\ast\bF=\bI\) and the first column of \(\bF\) is \(f_0=\tfrac{1}{\sqrt n}\mathbf 1\)), i.e.,
\[
\bA(\sigma)=\bF^\ast \Lambda(\sigma)\bF,
\]
with \(\Lambda(\sigma)=\operatorname{diag}\!\big(\mu_0(\sigma),\dots,\mu_{n-1}(\sigma)\big)\).

We write the first column of \(\bA(\sigma)\) as
\[
[\bA(\sigma)]_{j+1,1} = \frac{a_j(\sigma)}{G_0(\sigma)}, \qquad j=0,\dots,n-1,
\]
where \(a_j(\sigma)\) defines the \((j+1)\)-th unnormalized entry as a function of the kernel \(g_\sigma\), and
\[
G_0(\sigma) := \sum_{j=0}^{n-1} a_j(\sigma)
\]
is the normalizing factor (so that the first column sums to \(1\)). Note that \(a_0(\sigma)=1\) for kernels normalized at zero shift. To avoid normalization ambiguity we define the unnormalized discrete Fourier sums
\begin{equation*}
G_k(\sigma):=\sum_{j=0}^{n-1} a_j(\sigma)\,e^{-i\frac{2\pi k}{n}j},\qquad k=0,\dots,n-1,
\end{equation*}
and then, we can write the eigenvalues of $\bA(\sigma)$ as
\begin{equation*}
\mu_k(\sigma)=\frac{G_k(\sigma)}{G_0(\sigma)},\qquad k=0,\dots,n-1.
\end{equation*}
With this convention \(\mu_0(\sigma)=1\) for every \(\sigma\), and for \(k\neq0\) typically \(|\mu_k(\sigma)|<1\).

Using the spectral decomposition of \(\bA(\sigma)\) we can write the solution $\bx(\sigma)$ of
\[ \min_{\bx} \frac{1}{2}\|\bA(\sigma) \bx - \bb \|^2 + \frac{\lambda^2}{2}\| \bx\|^2\]
as
\begin{equation*}
\bx(\sigma)=\bF^*(|\Lambda|^2+\lambda^2\bI)^{-1}\Lambda^*\bF\bb,
\end{equation*}
where \(|\Lambda|^2\) denotes the componentwise squared-modulus of the diagonal entries. Define \(\widetilde{\bb}:=\bF\bb\). Then
\begin{equation*}
\|\bA(\sigma)\bx(\sigma)-\bb\|^2
=
\|(\Lambda(|\Lambda|^2+\lambda^2\bI)^{-1}\Lambda^*-\bI)\widetilde{\bb}\|^2
\end{equation*}
and
\begin{equation*}
\|\bx(\sigma)\|^2=\|(|\Lambda|^2+\lambda^2\bI)^{-1}\Lambda^*\widetilde{\bb}\|^2.
\end{equation*}
Combining the two terms yields the representation
\begin{equation}\label{eq:phi_fourier}
\varphi(\sigma)=\sum_{k=0}^{n-1}\frac{\lambda^2}{2\bigl(|\mu_k(\sigma)|^2+\lambda^2\bigr)}\,|\tilde b_k|^2,
\end{equation}
where \(\tilde b_k\) are the entries of \(\widetilde{\bb}\). The \(k=0\) (zero-frequency) term equals \(\dfrac{\lambda^2}{2(1+\lambda^2)}|\tilde b_0|^2\) and is independent of \(\sigma\). All \(\sigma\)-dependence arises from the \(k\neq0\) terms.

Differentiating \eqref{eq:phi_fourier} gives
\begin{equation}\label{eq:phi_prime}
\varphi'(\sigma)
=
-\sum_{k=1}^{n-1}\frac{\lambda^2\,\mathrm{Re}\!\big(\overline{\mu_k(\sigma)}\,\mu_k'(\sigma)\big)}{\big(|\mu_k(\sigma)|^2+\lambda^2\big)^2}\,|\tilde b_k|^2,
\end{equation}
since \(\dfrac{d}{d\sigma}|\mu_k(\sigma)|^2 = 2\,\mathrm{Re}(\overline{\mu_k(\sigma)}\mu_k'(\sigma))\).

We now state sufficient conditions under which \(\varphi(\sigma)\) admits \(\sigma=0\) as a global minimizer.

\begin{enumerate}[label=(\roman*)]

\item \emph{Removable singularity at \(\sigma=0\).} The limit \(\lim_{\sigma\to0}\bA(\sigma)\) exists and equals the identity matrix, so one may define \(\bA(0)=\bI\) by continuity.
\item \emph{Evenness in \(\sigma\).} The functions $a_j(\sigma)$ are differentiable and even, i.e., $a_j(\sigma) = a_j(-\sigma)$. Consequently \(\mu_k(\sigma)=\mu_k(-\sigma)\) and \(\varphi(\sigma)=\varphi(-\sigma)\); hence \(\varphi\) is even and \(\varphi'(0)=0\).
\item \emph{Modal sign condition.} For each nonzero frequency index \(k\),
\begin{equation*}
\mathrm{Re}\!\big(\overline{\mu_k(\sigma)}\,\mu_k'(\sigma)\big)>0
\qquad\text{for all }\sigma>0.
\end{equation*}
\end{enumerate}

If (i)–(iii) hold, then each summand in \eqref{eq:phi_prime} is strictly positive whenever \(|\tilde b_k|^2\neq0\). Hence, if 
\(\bb\) has nonzero energy in some nonzero-frequency mode (i.e., there exists \(k\ge1\) with \(|\tilde b_k|^2>0\)), then 
\(\varphi'(\sigma)>0\) for all \(\sigma>0\). By evenness of \(\varphi\), we then have \(\varphi'(\sigma)<0\) for all 
\(\sigma<0\). Therefore, the functional \(\varphi\) decreases on \((-\infty,0)\) and increases on \((0,\infty)\), and consequently \(\sigma=0\) is the unique global minimizer of \(\varphi\) on \(\mathbb{R}\).

The sufficient conditions (i)--(iii) above are precisely those exhibited by the
example analyzed in Section~\ref{sec: motivating example} for $n=2$.
In that setting, the normalization by $G_0(\sigma)=1+a_1(\sigma)$ ensures that
$\bA(\sigma)\to\bI$ as $\sigma\to0$, so that the forward operator admits a
continuous extension at $\sigma=0$, in agreement with condition~(i).
The functions $a_0(\sigma)=1$ and
$a_1(\sigma)=\exp\!\left(-\frac{1}{2\sigma^2}\right)$ are even.
Consequently, the nontrivial eigenvalue
$\mu_1(\sigma)=(1-a_1(\sigma))/(1+a_1(\sigma))$
is even, and the reduced functional $\varphi(\sigma)$ is even with
$\varphi'(0)=0$, corresponding to condition~(ii).
Finally, since there is only one nonzero Fourier mode in the $2\times2$ case,
condition~(iii) reduces to the sign of $\mu_1(\sigma)\mu_1'(\sigma)$.
As shown in Section~\ref{sec: motivating example}, this quantity is strictly
positive for all $\sigma>0$, implying that the reduced functional decreases for
$\sigma<0$ and increases for $\sigma>0$ whenever the data has nonzero energy in
the corresponding mode. Hence $\sigma=0$ is the unique global minimizer.
Overall, conditions (i)--(iii) provide a direct extension of those already present
in the $2\times2$ example.

\section*{Acknowledgments} The authors gratefully acknowledge the insightful comments and questions from Robert Kohn and Benjamin Seibold during seminar talks given by M.I. Espa\~nol at the Courant Institute of Mathematical Sciences at New York University and Temple University, respectively, which motivated this work. The authors also thank Mac Hyman for valuable discussions during his visit to Arizona State University in February 2024. M.I. Espa\~nol was supported through a Karen EDGE Fellowship. G. Jeronimo was partially supported by the Argentinian grants UBACYT 20020190100116BA and PIP 11220200101015CO CONICET. 

\bibliographystyle{plain}
\bibliography{biblio.bib}

\end{document}